\documentclass[leqno]{article}
\usepackage{amsmath, amsthm, amssymb}
\usepackage{mathrsfs}
\usepackage{graphicx}

\newtheorem{thm}{Theorem}[section]
\newtheorem{cor}[thm]{Corollary}
\newtheorem{lem}[thm]{Lemma}
\newtheorem{prop}[thm]{Proposition}

\theoremstyle{definition}
\newtheorem{defin}[thm]{Definition}
\newtheorem{rem}[thm]{Remark}
\newtheorem{exa}[thm]{Example}
\newtheorem{prob}[thm]{Problem}

\makeatletter
\@addtoreset{equation}{section} 

\makeatother

\begin{document}

%
%  -----------------------------
%       << Newcommand >>
%  -----------------------------
%
\newcommand{\BC}{\mathbb{C}}
\newcommand{\BR}{\mathbb{R}}
\newcommand{\BN}{\mathbb{N}}
\newcommand{\BZ}{\mathbb{Z}}

%
%  -------------------------------------------
%           << Title and Author >>
%  -------------------------------------------
%
%

\title{On the summability of formal solutions of
     some linear $q$-difference-differential equations}

\author{Hidetoshi \textsc{Tahara}}

\date{}

\maketitle

\begin{abstract}
   The paper considers the summability of formal
solutions $\hat{X}(t,z)=\sum_{n \geq 0}X_n(z)t^n$ of some analytic 
linear $q$-difference-differential equations in the complex 
domain: the equation is a $q$-difference 
equation with respect to the time variable $t$ and is a partial 
differential equation with respect to the space variables $z$.
The discussion is done by using a new framework of $q$-Laplace
and $q$-Borel transforms developed by the author.
\end{abstract}

{\it Key words and phrases}: $q$-difference-differential equations,
       summability, formal solutions, $q$-Gevrey asymptotics, 
       $q$-Laplace transforms.

{\it 2010 Mathematics Subject Classification Numbers}: 
       Primary 35C20; Secondary 35C10, 39A13.

\renewcommand{\thefootnote}{\fnsymbol{footnote}}

\footnote[0]{
       This work was supported by JSPS KAKENHI Grant Number
       JP15K04966.}

%
%
%
%  -----------------------------------------------------
%              << \S 1. Introduction >>
%  -----------------------------------------------------
%

\section{Introduction}\label{section 1}

   In Tahara \cite{fourier}, the author has introduced a 
new framework of $q$-Laplace and $q$-Borel transforms.
In this paper, we will apply its theory to the problem of the
summability of formal solutions of $q$-difference-differential 
equations of the form (\ref{1.1}) given below. 
The strategy of the argument was already explained in 
[\S 8, \cite{fourier}]: this paper gives a systematic study of
the problem of summability.
\par
   Let $q>1$ be fixed, and let 
$(t,z)=(t,z_1,\ldots,z_d) \in \BC \times \BC^d$ be the variables.
We define the $q$-difference operator $D_q$ in $t$ by
\[
    D_q(f(t,z))
    = \frac{f(qt,z)-f(t,z)}{qt-t}.
\]
For $\alpha=(\alpha_1,\ldots,\alpha_d) \in \BN^d$ 
(with $\BN=\{0,1,2,\ldots \}$) we write 
$|\alpha|=\alpha_1+\cdots+\alpha_d$ and 
$\partial_z^{\alpha}=\partial_{z_1}^{\alpha_1} 
                   \cdots \partial_{z_d}^{\alpha_d}$
with $\partial_{z_i}=\partial/\partial z_i$ ($i=1,\ldots,d$).
\par
    Let $m \in \BN^* (=\{1,2,\ldots \})$ and $\sigma>0$.
In this paper, we consider the linear $q$-difference-differential 
equation
\begin{equation}
    \sum_{j+\sigma |\alpha| \leq m} 
    a_{j,\alpha}(t,z)(tD_q)^j \partial_z^{\alpha}X = F(t,z)
    \label{1.1}
\end{equation}
under the following assumptions:
\par
   (1) $a_{j,\alpha}(t,z)$ ($j+\sigma |\alpha| \leq m$) and
$F(t,z)$ are holomorphic functions in a neighborhood of
$(0,0) \in \BC_t \times \BC_z^d$;
\par
   (2) (\ref{1.1}) has a formal power series solution
\begin{equation}
       \hat{X}(t,z)= \sum_{n \geq 0} 
       X_n(z) t^n \in {\mathcal O}_R[[t]] \label{1.2}
\end{equation}
where ${\mathcal O}_R$ (with $R>0$) denotes the set of all 
holomorphic functions on 
$D_R=\{z \in \BC^d \,;\, |z_i|<R \enskip (i=1,\ldots,d) \}$.

\par
\medskip
   As to the existence of such a formal solution of (1.1), 
see Remark \ref{Remark2.4}.
\par
   Our basic problem is:

\begin{prob}\label{Problem1.1}
   Under what condition can we get a true solution $W(t,z)$ of 
(\ref{1.1}) which admits $\hat{X}(t,z)$ as a $q$-Gevrey asymptotic 
expansion (in the sense of Definition \ref{Definition1.2} 
given below) ?
\end{prob}

   For $n \in \BN$ we write:
\[
   [n]_q= \frac{q^n-1}{q-1}, \quad [n]_q!=[1]_q[2]_q \cdots[n]_q.
\]
Of course, $[0]_q=0$ and $[0]_q!=1$. For 
$\lambda \in \BC \setminus \{0\}$ and $\epsilon >0$ we set
\begin{align*}
   &{\mathscr Z}_{\lambda}=\{-\lambda (q-1)q^m \,;\,
             m \in \BZ \}, \\
   &{\mathscr Z}_{\lambda, \epsilon}=
           \bigcup_{m \in \BZ} \{t \in \BC \,;\,
      |t+\lambda (q-1)q^m|<\epsilon |t| \}.
\end{align*}
We note that if $\epsilon >0$ is sufficiently small 
the set ${\mathscr Z}_{\lambda,\epsilon}$ is a disjoint union 
of closed disks. For $r>0$ we write $D_r^*=\{t \in \BC \,;\,
0<|t|<r \}$. Following Ramis-Zhang \cite{RZ} we define:

\begin{defin}\label{Definition1.2}
     (1) Let $\hat{X}(t,z)
      =\sum_{n \geq 0}X_n(z)t^n \in {\mathcal O}_{R}[[t]]$
and let $W(t,z)$ be a holomorphic function on 
$(D_r^* \setminus {\mathscr Z}_{\lambda}) \times D_R$
for some $r>0$. We say that {\it $W(t,z)$ admits 
$\hat{X}(t,z)$ as a $q$-Gevrey asymptitoc expansion 
on $(D_r^* \setminus {\mathscr Z}_{\lambda}) \times D_R$}, if 
there are $M>0$ and $H>0$ such that
\[
    \biggl| W(t,z)-\sum_{n=0}^{N-1} X_n(z)t^n \biggr|
      \leq \frac{M H^N}{\epsilon} [N]_q!|t|^N 
\]
holds on 
$(D_r^* \setminus {\mathscr Z}_{\lambda,\epsilon}) \times D_R$ 
for any $N=0,1,2,\ldots$ and any sufficiently small $\epsilon>0$.
\par
   (2) If there is a $W(t,z)$ as above, we say that the formal 
solution $\hat{X}(t,z)$ is $G_q$-summable in the direction 
$\lambda$.
\end{defin}

    This problem was already solved in 
Tahara-Yamazawa \cite{yama1, yama2} by using the framework of 
$q$-Laplace and $q$-Borel transforms developed by 
Ramis-Zhang \cite{RZ} and Zhang \cite{Z2}. 
In this paper, we will give a new proof by using $q$-Laplace 
and $q$-Borel transforms introduced in \cite{fourier}. 
\par
   Similar problems are discussed by Zhang \cite{Z1}, 
Marotte-Zhang \cite{MZ}, Ramis-Sauloy-Zhang \cite{RSZ} 
and Dreyfus \cite{dreyfus} in the 
$q$-difference equations, and by Malek \cite{M1,M2}, 
Lastra-Malek \cite{LM, LM2} and Lastra-Malek-Sanz \cite{LMS} 
in the case of $q$-difference-differential equations. But,
their equations are different from ours.
\par
   In this paper, we use the notations: $\BN=\{0,1,2,\ldots \}$
and $\BN^*=\{1,2,\ldots \}$. For an open set $W \subset \BC^d$
we denote by ${\mathcal O}(W)$ 
the set of all holomorphic functions on $W$. For an interval
$I=(\theta_1, \theta_2) \subset \BR$ we write
$S_I = \{ \xi \in {\mathcal R}(\BC \setminus \{0\}) \,;\,
      \theta_1< \arg \xi <\theta_2 \}$, where 
${\mathcal R}(\BC \setminus \{0\})$ denotes the universal 
covering space of $\BC \setminus \{0\}$.

%
%
%
%  -----------------------------------------------------
%              << \S 2. Main result >>
%  -----------------------------------------------------
%

\section{Main result}\label{section 2}

   For a holomorphic function $f(t,z)\, (\not\equiv 0)$ in a 
neighborhood of $(0,0) \in \BC_t \times \BC_z^d$, we define 
the order of the zeros of the function $f(t,z)$ at $t=0$ 
(we denote this by $\mathrm{ord}_t(f)$) by
\[
    \mathrm{ord}_t(f) = \min \{ p \in \BN \,;\, 
               (\partial_t^pf)(0,z) \not\equiv 0 
               \enskip \mbox{near $z=0$} \}.
\]
If $f(t,z) \equiv 0$ we set $\mathrm{ord}_t(f)=\infty$.
For $(a,b) \in \BR^2$ we write 
$C(a,b)=\{(x,y) \in \BR^2\,;\,
      x \leq a, y \geq b \}$. We set $C(a,\infty)=\emptyset$.
Then, the $t$-Newton polygon
$N_t(\ref{1.1})$ of equation (\ref{1.1}) is defined by
\[
    N_t(\ref{1.1})= \mbox{the convex hull of }
      \bigcup_{j+ \sigma |\alpha| \leq m} 
               C(j, \mathrm{ord}_t(a_{j,\alpha}))
\]
in $\BR^2$. Since we are considering (\ref{1.1}) under the 
assumption that (\ref{1.1}) has a formal solution $\hat{X}(t,z)$ 
in (\ref{1.2}), without loss of generality we may assume that 
\[
     \min \{\mathrm{ord}_t(a_{j,\alpha}) \,;\, 
                     j+\sigma|\alpha| \leq m \}=0.
\]
\par
   In this paper, we will consider the equation (\ref{1.1}) under 
the following conditions (A${}_1$), (A${}_2$) and (A${}_3$):
\par
\medskip
    (A${}_1$) There is an integer $m_0$ such that 
$0 \leq m_0<m$ and 
\[
        N_t(1.1)= \{(x,y) \in \BR^2 \,;\, x \leq m, \, 
                    y \geq \max\{0, x-m_0 \} \}.
\]
\par
   (A${}_2$) The following condition is satisfied:
\[
    |\alpha|>0 \Longrightarrow 
     (j, \mathrm{ord}_t(a_{j,\alpha})) \in int(N_t(\ref{1.1})),
\]
where $int(N_t(\ref{1.1}))$ denotes the interior of the set
$N_t(\ref{1.1})$ in $\BR^2$. 
\par
   (A${}_3$) In addition, we have
\[
    a_{m_0,0}(0,0) \ne 0, \quad 
    \frac{a_{m,0}(t,z)}{t^{m-m_0}} \Bigr|_{t=0,z=0} \ne 0.
\]

\par
\medskip
   By (A${}_1$), we have $a_{m,0}(t,z)=O(t^{m-m_0})$ 
(as $t \longrightarrow 0$), and so the 
second condition in (A${}_3$) makes sense.
\par
   The figure of $N_t(\ref{1.1})$ is as in Figure 1. In Figure 1, 
the boundary of $N_t(\ref{1.1})$ consists of a horizontal half-line
$\Gamma_0$, a segment $\Gamma_1$ and a vertical
half-line $\Gamma_2$, and $k_i$ is the slope of $\Gamma_i$ 
($i=0,1,2$).  By (A${}_1$) we have $k_1=1$.

\begin{figure}[htbp]
\begin{center}
%%%%%%%%%%%%%%%%%%%%%%%%%%%%%%%%%%%%%%%%%%%%%%%%%%%%%%%%%
%
\qquad
\includegraphics[scale=0.7]{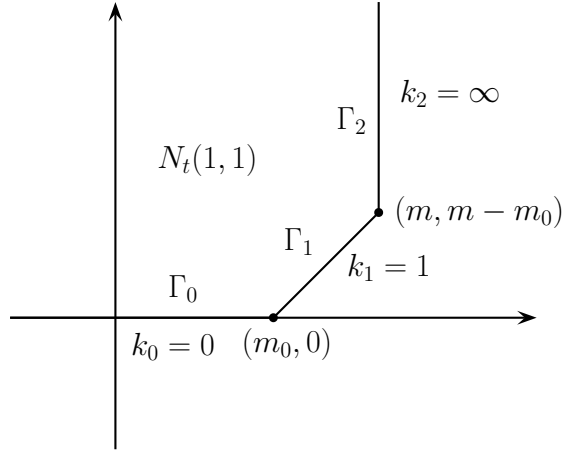}
%
%%%%%%%%%%%%%%%%%%%%%%%%%%%%%%%%%%%%%%%%%%%%%%%%%%%%%%%%%
%
\end{center}
\vspace{-3mm}
\caption{The $t$-Newton polygon of (1.1)} \label{Fig1}
\end{figure}

We note:

\begin{lem}\label{Lemma2.1} 
     By {\rm (A${}_1$)} and {\rm (A${}_2$)} we have
\begin{equation}
    \mathrm{ord}_t(a_{j,\alpha}) \geq 
    \left\{ \begin{array}{ll}
     \max \{0, j-m_0 \}, &\mbox{if $|\alpha|=0$}, \\
     \max \{1, j-m_0+1 \}, &\mbox{if $|\alpha|>0$}.
           \end{array}  \right.   \label{2.1}
\end{equation}
\end{lem}

   By the assumption, $a_{j,0}(t,z)$ ($m_0 \leq j \leq m$) 
can be expressed in the form
\[
    a_{j,0}(t,z)=t^{j-m_0} b_{j,0}(t,z) 
\]
for some holomorphic functions $b_{j,0}(t,z)$ 
($m_0 \leq j \leq m$) satisfying $b_{m_0,0}(0,0)$ $\ne 0$ and 
$b_{m,0}(0,0) \ne 0$.  We set
\[
     P_0(\lambda,z)= \sum_{m_0 \leq j \leq m}
            \frac{b_{j,0}(0,z)}{q^{j(j-1)/2}} \lambda^{j-m_0}
\]
and denote by $\lambda_1, \ldots, \lambda_{m-m_0}$ the
roots of $P_0(\lambda,0)=0$.  Since $b_{m_0,0}(0,0) \ne 0$, 
we have $\lambda_i \ne 0$ for all $i=1,2,\ldots,m-m_0$. We set
\[
     S = \bigcup_{i=1}^{m-m_0}
       \{\xi= \lambda_i \eta \,;\, \eta >0 \} 
      \enskip  \subset \enskip \BC_{\xi}
\]
which is a candidate of the set of singular directions at $z=0$.
The role of the set $S$ lies in

\begin{lem}\label{Lemma2.2} 
   For any $\lambda \in \BC \setminus(\{0\} \cup S)$ we can find 
a $\delta >0$, an interval $I=(\theta_1, \theta_2)$ with 
$\theta_1<\arg \lambda <\theta_2$ and an $R>0$ such that
$|P_0(\xi,z)| \geq \delta (1+|\xi|)^{m-m_0}$ holds on 
$S_I \times D_R$.
\end{lem}

   The purpose of this paper is to prove the following result.

\begin{thm}[Main theorem]\label{Theorem2.3}
    Suppose {\rm (A${}_1$)}, {\rm (A${}_2$)}, 
{\rm (A${}_3$)} and the additional condition
\begin{equation}
    \mathrm{ord}_t(a_{j,\alpha}) \geq j-m_0+2, 
     \quad \mbox{if $m_0 \leq j<m$ and $|\alpha|>0$}.
     \label{2.2}
\end{equation}
Let 
$\hat{X}(t,x)= \sum_{n \geq 0}X_n(z)t^n \in {\mathcal O}_R[[t]]$ 
be a formal solution of {\rm (\ref{1.1})}.  Then, for any 
$\lambda \in \BC \setminus(\{0\} \cup S)$ there are 
$r>0$, $R_1>0$, a holomorphic solution $W(t,z)$ of {\rm (\ref{1.1})} on 
$(D_r^* \setminus {\mathscr Z}_{\lambda}) \times D_{R_1}$ such that
$W(t,z)$ admits 
$\hat{X}(t,z)$ as a $q$-Gevrey asymptitoc expansion 
on $(D_r^* \setminus {\mathscr Z}_{\lambda}) \times D_{R_1}$.
\end{thm}

\begin{rem}\label{Remark2.4}
   (1) Set
\[
     P_1(\lambda;z)= \sum_{0 \leq j \leq m_0}a_{j,0}(0,z) \lambda^j.
\]
If $P_1([n]_q;0) \ne 0$ holds for any $n \in \BN$, the equation 
(\ref{1.1}) has a unique formal solution 
$\hat{X}(t,x) \in {\mathcal O}_R[[t]]$ for some $R>0$.
\par
   (2) The additional condition (\ref{2.2}) seems to be a little bit 
strange, but as is seen in the proof of Theorem \ref{Theorem2.3} in
\S 5 we need to use this condition. At present, the author does not 
know how to remove it.
\par
   (3) In the case where the condition (\ref{2.2}) is not satisfied, 
as is seen in \S 6, by setting $q_1=q^{1/4}$ and $t=\tau^2$ we can 
trasnform (\ref{1.1}) to 
\begin{equation}
    \sum_{j+\sigma |\alpha| \leq m}
            A_{j,\alpha}(\tau,z) 
           \biggl( \frac{1}{[4]_{q_1}}
           \Bigl( (q_1-1)(\tau D_{q_1})^2 
      +  2(\tau D_{q_1}) \Bigr) \biggr)^j \partial_z^{\alpha}Y
     = G(\tau,z) \label{2.3}
\end{equation}
where 
\begin{align*}
    &A_{j,\alpha}(\tau,z)= a_{j,\alpha}(\tau^2,z) \quad
            (j+\sigma |\alpha| \leq m), \\
    &Y(\tau,z)
          =X(\tau^2,z) = \sum_{n \geq 0} X_n(z) \tau^{2n}, \\
    &G(\tau,z)= F(\tau^2,z).
\end{align*}
Since this equation (\ref{2.3}) satisifies (\ref{2.2}), we can apply 
Theorem \ref{Theorem2.3} to (\ref{2.3}), and obtain the 
$G_q$-summability of $Y(\tau,z)$. Thus, the constraint by 
(\ref{2.2}) is not a big problem. For details, see \S 6.
\end{rem}

\begin{exa}\label{Example2.5}
    Let us consider
\begin{equation}
    a t (t^2D_q)^2X+b(tD_q)X+cX
          + t^{n_1} \partial_z^{\alpha_1}(tD_q)X
          + t^{n_0} \partial_z^{\alpha_0}X =F(t,z)
    \label{2.4}
\end{equation}
where $a \ne 0$, $b \ne 0$, $c \in \BC$, 
$n_i \in \BN^*$ ($i=0,1$) and $\alpha_i \in \BN^*$ ($i=0,1$).
Then, this equation satisfies {\rm (A${}_1$)}, {\rm (A${}_2$)} and 
{\rm (A${}_3$)} with $m_0=1$, $m=2$ and $k_1=1$. We note that 
{\rm (A${}_2$)} coresponds to the condition ``$n_1 \geq 1$ and 
$n_0 \geq 1$", and that {\rm (A${}_3$)} coresponds to the condition 
``$a \ne 0$ and $b \ne 0$".  In this case, (\ref{2.2}) corresponds to 
the condition $n_1 \geq 2$.  Thus, if $n_1 \geq 2$ we can apply 
Theorem \ref{Theorem2.3} to (\ref{2.4}).
\end{exa}

    We note that Theorem \ref{Theorem2.3} is already proved in 
\cite{yama1}. The purpose of this paper is to give a new proof in 
the framework of $q$-Laplace and $q$-Borel transforms given 
in \cite{fourier}. This new proof produces various new tools 
and techniques which will be very useful in treating other problems, 
and by this reason, the author believes that it is worthy to write
this paper.
\par
   The rest part of this paper is organized as follows.
In the next \S 3, we summarize basic results of $q$-Laplace
and $q$-Borel transforms in \cite{fourier}. In \S 4, we do
some preparatory discussins which are needed in the proof
of Theorem \ref{Theorem2.3}. In \S 5, we prove 
Theorem \ref{Theorem2.3} by using a result in \S 4.. In the 
last \S 6, we discuss the case without the condition (\ref{2.2}).

%
%
%
%  -----------------------------------------------------
%    << \S 3. $q$-Laplace and $q$-Borel transforms >>
%  -----------------------------------------------------
%

\section{$q$-Laplace and $q$-Borel transforms}\label{section 3}

   In this section, we summarize basic results on $q$-Laplace 
and $q$-Borel transforms developed in \cite{fourier} with small
modifications.  We always suppose: $q>1$.

%
%
% --------------------------------------------------
%    << 3.1. $q$-Laplace transforms  >>
% --------------------------------------------------
%

\subsection{$q$-Laplace transforms}\label{subsection 3.1}

   Let $\lambda \in \BC \setminus \{0\}$,
and set $\lambda q^{\BZ}=\{\xi= \lambda q^m \,;\, m \in \BZ \}$.
For a function $f(\xi,z)$ on $\lambda q^{\BZ} \times D_R$, 
we define the $q$-Laplace transform 
$F(t,z)={\mathscr L}_q^{\lambda}[f](t,z)$ of $f(\xi,z)$ in
the direction $\lambda$ by
\begin{equation}
   {\mathscr L}_q^{\lambda}[f](t,z)
    = \int_{\lambda q^{\BZ}}
    \textrm{Exp}_q(-q \xi/t)\, f(\xi,z) d_q\xi, \label{3.1}
\end{equation}
where 
\[
   \textrm{Exp}_q(x)
     =\sum_{n \geq 0} \frac{q^{n(n-1)/2}}{[n]_q!}x^n
     = \frac{1}{\displaystyle \prod_{m \geq 0}(1-q^{-m-1}(q-1)x)}
\]
(the first equality is the definition of $\textrm{Exp}_q(x)$ and 
the second equality is from Euler's indentity), 
and the integral in (\ref{3.1}) is taken in the following sense: 
\[
    \int_{\lambda q^{\BZ}} g(\xi) d_q \xi
    = \sum_{m \in \BZ} g(\lambda q^m) (\lambda q^{m+1}- \lambda q^m),
\]
which is a discretization of the classical integral.
\par
    In the case $\lambda=1$ we write ${\mathscr L}_q$ instead of 
${\mathscr L}_q^{1}$. If we define the operator $M_{\lambda}$ by 
$M_{\lambda}[f](\xi)=f(\lambda \xi)$ we have 
${\mathscr L}_q^{\lambda}= 
  \lambda(M_{1/\lambda} \circ {\mathscr L}_q \circ M_{\lambda})$.
Since ${\mathscr L}_q$ is investigated quite well in 
\cite{fourier}, we have the following properties (see Example 3.1, 
Proposition 3.2 and Proposition 7.1 in \cite{fourier}).

\par
\medskip
    (3-1-1) If $f=\xi^n$ ($n \in \BN$), 
we have ${\mathscr L}_q^{\lambda}[f] = [n]_q! t^{n+1}$.
\par
    (3-1-2) If $f(\xi,z)$ is a function on 
$\lambda q^{\BZ} \times D_R$ which is holomorphic in $z \in D_R$ 
and if
\begin{align*}
    &|f(\lambda q^n,z)| \leq Ch^n [n]_q! 
            \enskip \mbox{on $D_R$}, \enskip n=0,1,2,\ldots, \\
    &|f(\lambda q^{-m},z)| \leq AB^m 
              \enskip \mbox{on $D_R$}, \enskip m=1,2,\ldots
\end{align*}
for some $C>0$, $h>0$, $A>0$ and $0<B<q$, 
the function $F(t,z)={\mathscr L}_q^{\lambda}[f](t,z)$
is well-defined as a holomorphic function on 
$(D_r^* \setminus {\mathscr Z}_{\lambda}) \times D_R$ for 
a sufficiently small $r>0$.
\par
    (3-1-3) In addition, $F(t,z)$ has at most simple poles on 
${\mathscr Z}_{\lambda}$ with respect to $t$, and there is an
$H>0$ such that
\[
    |F(t,z)| \leq \frac{H}{\epsilon} |t|^{\alpha}
     \quad \mbox{on $(D_r^* \setminus 
                 {\mathscr Z}_{\lambda, \epsilon}) \times D_R$}
\]
for any $\epsilon >0$, where 
$\alpha=(\log q-\log B)/\log q$ ($>0$).
\par
   (3-1-4) The following result gives a Watson type lemma.

\begin{prop}[Watson type lemma]\label{Proposition 3.1}
   Let $f(\xi,z)$ be a function on $\lambda q^{\BZ} \times D_R$, 
and let $c_k(z)$ {\rm (}$k=0,1,2,\ldots${\rm )} be functions on $D_R$. 
Suppose that there are $C>0$, $h>0$, $A>0$ and $h_1>0$ such that
\begin{align}
    &|f(\lambda q^n,z)| \leq C h^n [n]_q! \enskip
            \mbox{on $D_R$}, \enskip n=0,1,2, \cdots, \label{3.2}\\
    &\Bigl| f(\xi,z)- \sum_{k=0}^{N-1} 
                     c_k(z) \xi^k \Bigr|
             \leq A{h_1}^N |\xi|^N \enskip
             \mbox{on $D_R$}  \label{3.3}\\
    &\qquad \qquad \quad 
        \mbox{for $\xi=\lambda q^{-m}$ {\rm (}$m=1,2, \cdots${\rm )} and 
                               $N=0,1,2,\ldots$}. \notag
\end{align}
Then, there are $M>0$ and $H>0$ such that
\[
    \Bigl| {\mathscr L}_q^{\lambda}[f](t,z) -
        \sum_{k=0}^{N-1} c_k(z) [k]_q! t^{k+1} \Bigr|
        \leq \frac{MH^N}{\epsilon} [N]_q! |t|^{N+1}
\]
on $(D_r^* \setminus {\mathscr Z}_{\lambda, \epsilon}) 
       \times D_R$ for any $\epsilon >0$ and $N=0,1,2,\ldots$.
\end{prop}

%
%
% --------------------------------------------------
%    << 3.2. $q$-Borel transforms >>
% --------------------------------------------------
%

\subsection{$q$-Borel transforms}\label{subsection 3.2}

    For a holomorphic function $F(t,z)$ on 
$(D_r^* \setminus {\mathscr Z}_{\lambda}) \times D_R$ having at 
most simple poles on ${\mathscr Z}_{\lambda}$ with respect to $t$, 
we define the $q$-Borel transform $f(\xi,z)={\mathscr B}_q[F](\xi,z)$ 
of $F(t,z)$ in the direction $\lambda$ by
\begin{equation}
    {\mathscr B}_q^{\lambda}[F](\xi,z)= \frac{1}{2\pi i} 
       \int_{|t|=\rho_{\xi}} F(t,z) \textrm{exp}_q(\xi/t) 
             \frac{dt}{t^2}, \quad
       \xi=\lambda q^k \enskip (k \in \BZ) \label{3.4}
\end{equation}
where $\rho_{\xi}>0$ is sufficiently small depending on $\xi$,
\[
   \textrm{exp}_q(x)=\sum_{n \geq 0} \frac{1}{[n]_q!}x^n
       = \prod_{m \geq 0}(1+q^{-m-1}(q-1)x)
\]
(the first equality is the definition of $\textrm{exp}_q(x)$ and 
the second equality is from Euler's indentity), and the integral
in (\ref{3.4}) is taken as a contour integral along the circle
$\{t \in \BC \,;\, |t|=\rho_{\xi} \}$ in the complex plane.
\par
    In the case $\lambda=1$ we write ${\mathscr B}_q$ instead of
${\mathscr B}_q^{1}$.  Since we have 
${\mathscr B}_q^{\lambda}= 
  (M_{1/\lambda} \circ {\mathscr B}_q \circ M_{\lambda})/\lambda$ 
holds, and since ${\mathscr B}_q$ is investigated quite well 
in \cite{fourier}, we have the following properties (see Example 4.1,
Proposition 4.2, Theorem 5.1 and Theorem 5.4 in \cite{fourier}).
\par
\medskip
   (3-2-1) If $F(t)=t^{n+1}$ ($n \in \BN$) we have 
${\mathscr B}_q^{\lambda}[F]= \xi^n/[n]_q!$.
\par
   (3-2-2) If $F(t,z)$ satisfies
\[
    |F(t,z)| \leq \frac{H}{\epsilon} |t|^{\alpha}
     \quad \mbox{on $(D_r^* \setminus 
             {\mathscr Z}_{\lambda, \epsilon}) \times D_R$},
       \quad \forall \epsilon >0
\]
for some $H>0$ and $\alpha>0$, the function
$f(\xi,z)={\mathscr B}_q^{\lambda}[F](\xi,z)$ is well-defined as a 
function on $\lambda q^{\BZ} \times D_R$ and it is holomorphic in 
$z \in D_R$.
\par
   (3-2-3) The following gives inversion formulas.

\begin{thm}[Inversion formulas]\label{Theorem3.2}
   {\rm (1)} If $f(\xi,z)$ satisfies the assumption in {\rm (3-1-2)}, 
we have
\[
     f(\xi,z)
     = ({\mathscr B}_q^{\lambda} \circ
        {\mathscr L}_q^{\lambda})[f](\xi,z)
    \quad \mbox{on $\lambda q^{\BZ} \times D_R$}.
\]
\par
   {\rm (2)} If $F(t,z)$ satisfies the assumption in {\rm (3-2-2)}, 
we have
\[
     F(t,z)
     = ({\mathscr L}_q^{\lambda} \circ
                 {\mathscr B}_q^{\lambda})[F](t,z)
    \quad \mbox{on 
    $(D^*_{r_1} \setminus {\mathscr Z}_{\lambda}) \times D_R$}.
\]
for some $r_1>0$.
\end{thm}

    By (3-2-1), it will be reasonable to define the formal 
$q$-Borel transform $\hat{\mathscr B}_q$ in the following way:
\[
    \hat{\mathscr B}_q[\hat{F}](\xi,z)=\sum_{n \geq 0}
             \frac{a_n(z)}{[n]_q!}\xi^n \quad 
   \mbox{for} \enskip
     \hat{F}(t,z)=\sum_{n \geq 0}a_n(z)t^{n+1}.
\]
Since $D_q[t^n]=[n]_q t^{n-1}$ holds, this formal $q$-Borel
transform is just fitting to our equation (\ref{1.1}).

%
%
% --------------------------------------------------
%         << 3.3. $q$-Convolutions >>
% --------------------------------------------------
%

\subsection{$q$-Convolutions}\label{subsection 3.3}

    Let $a(\xi,z)= \sum_{k \geq 0} a_k(z) \xi^k$ be a holomorphic 
function in a neighborhood of the origin of 
$\BC_{\xi} \times \BC_z^d$. For a function $f(\xi,z)$ we define 
the $q$-convolution $(a*_qf)(\xi,z)$ of $a(\xi,z)$ and $f(\xi,z)$ 
with respect to $\xi$ by
\begin{equation}
    (a*_qf)(\xi,z)= \sum_{k \geq 0}\frac{a_k(z)}{q^k}
        \int_0^{\xi}(\xi-py)^k_p f(q^{-k-1}\xi,z) d_p\xi,
    \label{3.5}
\end{equation}
where $p=1/q$, the integral in (\ref{3.5}) is taken as $p$-Jackson 
integral, and $(\xi-py)^k_p$ is defined by the following: 
$(\xi-py)^0_p=1$ and for $k \geq 1$
\[
    (\xi-py)^k_p=(\xi-py)(\xi-p^2y) \cdots (\xi-p^ky).
\]
\par
   We have the following properties:
\par
\medskip
   (3-3-1) By Example 6.1 in \cite{fourier} we have
\[
      \xi^m *_q \xi^n = \dfrac{[m]_q![n]_q!}{[m+n+1]_q!}
              \xi^{m+n+1} \quad \mbox{for any $m,n \in \BN$}. 
\]
\par
   (3-3-2)    For $I=(\theta_1, \theta_2) \subset \BR$ and 
$0<r \leq \infty$, we write
\begin{align*}
    &S_I = \{\xi \in {\mathcal R}(\BC \setminus \{0\}) \,;\,
       0<|\xi|<\infty, \theta_1< \arg \xi < \theta_2 \}, \\
    &S_I(r)= \{\xi \in {\mathcal R}(\BC \setminus \{0\}) \,;\,
       0<|\xi|<r, \theta_1< \arg \xi < \theta_2 \}.
\end{align*}
Then, by the definition of $q$-convolution we have

\begin{lem}\label{Lemma3.3}
    {\rm (1)} Let $a(\xi,z)$ be a holomorphic 
function on $D_r \times D_R$, and let $f(\xi,z)$ be a 
holomorphic function on $S_I(r) \times D_R$ satisfying the 
following: there is an 
$0<\alpha<1$ such that $\xi^{\alpha}f(\xi,z)$ is bounded on 
$S_I \times D_R$. Then, $(a *_q f)(\xi,z)$ is well-defined as a 
holomorphic function on $S_I(qr) \times D_R$.
\par
   {\rm (2)} In addition, if a holomorphic function 
$A(\xi)=\sum_{k \geq 0}A_k\xi^k$ on $D_r$ and a continuous 
function $F(x)$ on $x \geq 0$ satisfy
\begin{align*}
   &1) \enskip a(\xi,z) \ll A(\xi) \quad 
             \mbox{in $\BC[[\xi]]$ for any $z \in D_R$}, \\
   &2) \enskip |f(\xi,z)| \leq F(|\xi|) 
                     \quad \mbox{on $S_I(r) \times D_R$}
\end{align*}
{\rm (}where $\sum_{k \geq 0}a_k\xi^k \ll \sum_{k \geq 0}b_k\xi^k$ 
in $\BC[[\xi]]$ means 
that $|a_k| \leq b_k$ holds for all $k \geq 0${\rm )}, we have
\[
    |(a *_q f)(\xi,z)| \leq (A *_q F)(|\xi|) \quad
        \mbox{on $S_I(qr) \times D_R$}.
\]
\end{lem}

\smallskip
   (3-3-3) By Theorems 6.3 and 6.7 in \cite{fourier} we have

\begin{thm}[Convolution theorem]\label{Theorem3.4}
   {\rm (1)} Let $f(\xi,z)$ 
be a function on $\lambda q^{\BZ} \times D_R$ satisfying the 
condition in {\rm (3-1-2)}, and let $a(\xi,z)$ be a holomorphic 
function on $\BC \times D_R$ with the estimate 
\begin{equation}
    |a(\xi,z)| \leq M |\xi|^{\alpha} 
       \exp \Bigl( \dfrac{(\log |\xi|)^2}{2 \log q} \Bigr)
       \quad
      \mbox{on $(\BC_{\xi} \setminus \{0\}) \times D_R$}
              \label{3.6}
\end{equation}
for some $M>0$ and $\alpha \in \BR$.  Then, we have
\[
      {\mathscr L}_q^{\lambda}[a*_q f](t,z)
     = {\mathscr L}_q^{\lambda}[a](t,z) 
                    \times {\mathscr L}_q^{\lambda}[f](t,z)
     \quad \mbox{on $(D^*_r \setminus {\mathscr Z}_{\lambda})
                     \times D_R$}
\]
for some $r>0$.
\par
   {\rm (2)} Let $A(t,z)$ be a holomorphic function on 
$D_r \times D_R$ satisfying $A(t,z)=O(|t|)$ 
{\rm (}as $|t| \longrightarrow 0$ uniformly 
on $D_R${\rm )}, and let $F(t,z)$ be a holomorphic function on 
$(D^*_r \setminus {\mathscr Z}_{\lambda}) \times D_R$ having at 
most simple poles on the set ${\mathscr Z}_{\lambda}$ with respect 
to $t$. Suppose the condition in {\rm (3-2-2)}.  Then, we have 
\[
    {\mathscr B}_q^{\lambda}[A \times F](\xi,z)
    = ({\mathscr B}_q^{\lambda}[A] 
             *_q {\mathscr B}_q^{\lambda}[F])(\xi,z) \quad
   \mbox{on $\lambda q^{\BZ} \times D_R$}.
\]
\end{thm}

   As to the estimate of type (\ref{3.6}), we have the following 
result (see Proposition 2.1 in \cite{ramis}):

\begin{prop}\label{Proposition3.5}
    Let $\hat{f}(\xi)= \sum_{n \geq 0}a_n\xi^n
\in \BC[[\xi]]$. The following two conditions are equivalent:
\par
   {\rm (1)} There are $A>0$ and $H>0$ such that
\[
         |a_n| \leq \dfrac{A H^n}{[n]_q!}, 
         \quad n=0,1,2,\ldots. 
\]
\par
   {\rm (2)} $\hat{f}(\xi)$ is the Taylor expansion at $\xi=0$ of
an entire function $f(\xi)$ satisfying the estimate
\[
     |f(\xi)| \leq M |\xi|^{\alpha}
        \exp \Bigl( \dfrac{(\log |\xi|)^2}{2 \log q} \Bigr)
     \quad \mbox{on $\BC_{\xi} \setminus \{0\}$} 
\]
for some $M>0$ and $\alpha \in \BR$.
\end{prop}

%
%
% --------------------------------------------------
%         << 3.4. Some other results >>
% --------------------------------------------------
%

\subsection{Some other results}\label{subsection 3.4}

   In the application to $q$-difference equations, we need some 
more results. We summarize such results here.
\par
\medskip
   (3-4-1) If 
$a(\xi,z)=\sum_{k \geq 0} a_k(z)\xi^k$ and 
$f(\xi,z)=\sum_{i \geq 0} f_i(z)\xi^i$ 
are holomorphic functions on $D_r \times D_R$, the 
$q$-convolution $(a*_qf)(\xi,z)$ is well-defined as a holomorphic
function on $D_{rq} \times D_R$, and its Taylor 
expansion is given by
\begin{equation}
      (a*_qf)(\xi,z) =\sum_{n \geq 0}\biggl(
        \sum_{k+i=n} a_k(z)f_i(z) 
      \dfrac{[k]_q![i]_q!}{[k+i+1]_q!} \biggr) \, \xi^{n+1}
     \label{3.7}
\end{equation}
(by Proposition 6.2 in \cite{fourier}).
\par
   (3-4-2) By (\ref{3.7}), it will be reasonable to define the 
formal $q$-convolution $(a \hat{*}_qf)(\xi,z)$ of two
series $a(\xi,z)=\sum_{k \geq 0} a_k(z)\xi^k$ and 
$f(\xi,z)=\sum_{i \geq 0} f_i(z)\xi^i$ in ${\mathcal O}_R [[\xi]]$
by
\[
      (a \hat{*}_qf)(\xi,z) =\sum_{n \geq 0}\biggl(
        \sum_{k+i=n} a_k(z)f_i(z) 
            \dfrac{[k]_q![i]_q!}{[k+i+1]_q!} \biggr) \, \xi^{n+1}.
\]
We have:

\begin{lem}\label{Lemma3.6}
   For two formal 
series $A(t,z)$ and $W(t,z)$ in ${\mathcal O}_R [[t]] \times t$
we have
\begin{equation}
     \hat{\mathscr B}_q[ A \times W](\xi,z)
      =\hat{\mathscr B}_q[A](\xi,z)
        \hat{*}_q \hat{\mathscr B}_q[W](\xi,z). \label{3.8}
\end{equation}
\end{lem}

\begin{proof}
    Since the summations in (\ref{3.8}) are formal, to prove
(\ref{3.8}) it is enough to show (\ref{3.8}) in the case 
$A(t,z)=a_k(z) t^{k+1}$ and $W(t,z)=w_i(z) t^{i+1}$. In this 
case, we have $(A \times W)(t,z)=a_k(z)w_i(z)t^{(k+i+1)+1}$ 
and so 
$$
    \hat{\mathscr B}_q[ A \times W](\xi,z)
       = \frac{a_k(z)w_i(z)}{[k+i+1]_q!} \xi^{k+i+1}.
$$
On the other hand, we have
\begin{align*}
   \hat{\mathscr B}_q[A](\xi,z)
        \hat{*}_q \hat{\mathscr B}_q[W](\xi,z)
   &= \Bigl( \frac{a_k(z)}{[k]_q!} \xi^k \Bigr) \, \hat{*}_q
      \Bigl( \frac{w_i(z)}{[i]_q!} \xi^i \Bigr) \\
   &= \frac{a_k(z)}{[k]_q!}\frac{w_i(z)}{[i]_q!}
     \times \frac{[k]_q! [i]_q!}{[k+i+1]_q!} \xi^{k+i+1} \\
   &= \frac{a_k(z)w_i(z)}{[k+i+1]_q!} \xi^{k+i+1}.
\end{align*}
Hence we have (\ref{3.8}).
\end{proof}

\par
\smallskip
   (3-4-3) If $A(t,z)$ ($\in {\mathcal O}_R [[t]] \times t$) is 
convergent in a neighborhood of $(0,0) \in \BC_t \times \BC_z^d$, 
we have
$\hat{\mathscr B}_q[A](\xi,z)={\mathscr B}_q^{\lambda}[A](\xi,z)$ 
for any $\lambda \in \BC \setminus \{0\}$. If 
$w(\xi,z)= \hat{\mathscr B}_q[W](\xi,z)$ is convergent in a 
neighborhood of $(0,0) \in \BC_\xi \times \BC_z^d$, the right-hand 
side of the formula (\ref{3.8}) is expressed in the form
\begin{align}
     \hat{\mathscr B}_q[A](\xi,z)
        \hat{*}_q \hat{\mathscr B}_q[W](\xi,z)
    &= {\mathscr B}_q[A](\xi,z) *_q w(\xi,z) \label{3.9}\\
    &= {\mathscr B}_q^{\lambda}[A](\xi,z) *_q w(\xi,z)
               \notag
\end{align}
for any $\lambda \in \BC \setminus \{0\}$.
\par
   (3-4-4) Let $X(t,z) \in {\mathcal O}_R [[t]]$ with 
$X(0,x) \equiv 0$ and set
$u(\xi,z)=\hat{\mathscr B}_q[X]$ $(\xi,z)$: then we have
\[
      \hat{\mathscr B}_q[(t^2D_q)^iX](\xi,z)= \xi^i u(\xi,z),
    \quad i=1,2,\ldots.
\]

%
%
% --------------------------------------------------
% << \S 4. On a $q$-convolution equation >>
% --------------------------------------------------
%

\section{On a $q$-convolution equation}\label{section 4}

   The main part of the proof of Theorem \ref{Theorem2.3} consists 
of the analysis of a $q$-convolution equation which is obtained by 
applying the formal $q$-Borel transform to (\ref{1.1}). Hence, 
in this section we discuss only $q$-convolution
equations first.
\par
    Let $q>1$, $I=(\theta_1,\theta_2)$ be a non-empty open interval,
$0 <r \leq \infty$ and $R>0$. We set
\[
     \phi_m(x;h) 
        = \sum_{i=0}^{\infty} \frac{h^ix^{m+i}}{[m+i]_q!},
    \quad m=0,1,2,\ldots \enskip \mbox{and} \enskip
          h>0.
\]
Let us consider the 
$q$-convolution partial differential equation

\begin{align}
    P(\xi,z)u &+ \sum_{0 \leq i \leq m}
        c_{i,0}(\xi,z) *_q (\xi^i u) \label{4.1}\\
   &+ \sum_{i+\sigma |\alpha| \leq m, |\alpha|>0}
     c_{i,\alpha}(\xi,z) *_q (1 *_q (\xi^i \partial_z^{\alpha}u))
        = f(\xi,z) \notag
\end{align}
under the following assumptions:
\par
\medskip
   h${}_1$) \quad $\sigma >0$, $m_0 \in \BN$, $m \in \BN^*$,
         $0 \leq m_0<m$ and $0<R \leq 1$;
\par
   h${}_2$) \quad $P(\xi,z)$ is a polynomial in $\xi$ with 
ccoefficients in ${\mathcal O}(D_R)$, and there is a 
$\delta>0$ such that 
$|P(\xi,x)| \geq \delta |\xi|^{m_0}(1+|\xi|)^{m-m_0}$ holds 
on $S_I(r) \times D_R$.
\par
   h${}_3$) \quad 
    $c_{i,\alpha}(\xi,z) \in {\mathcal O}(\BC \times D_R)$
($i+\sigma |\alpha| \leq m$) and there are $C_{i,\alpha}>0$ 
($i+\sigma |\alpha| \leq m$) and $h_0>0$ such that
\begin{align*}
    &c_{i,\alpha}(\xi,z) \ll C_{i,\alpha}\phi_{m_0-i-1}(\xi;h_0),
          \quad \mbox{if $0 \leq i<m_0$}, \\
    &c_{i,\alpha}(\xi,z) \ll C_{i,\alpha}\phi_{0}(\xi;h_0),
          \quad \mbox{if $m_0 \leq i \leq m$}
\end{align*}
hold (as formal power series in $\xi$) for any $z \in D_R$.

\par
\medskip
   Then, we have

\begin{prop}\label{Proposition4.1}
    Suppose the conditions {\rm h${}_1$)}, {\rm h${}_2$)}
and {\rm h${}_3$)}.  Let 
$f(\xi,z) \in {\mathcal O}(S_I(r) \times D_R)$ and suppose the 
estimate
\[
             |f(\xi,z)| \leq B \phi_N(|\xi|;h)
        \quad \mbox{on $S_I(r) \times D_R$}  
\]
for some $B>0$, $h> h_0$ and some $N \in \BN^*$ satisfying 
$N \geq m_0$ and
\[
     \beta= \frac{1}{[N]_q} \sum_{0 \leq i<m_0}
           \frac{C_{i,0}}{\delta (1-h_0/h)} <1.
\]
Then, the equation {\rm (\ref{4.1})} has a unique solution 
$u(\xi,x) \in {\mathcal O}(S_I(r) \times D_R)$ which satisfies 
the following estimate: for any $0<R_1<R$ there are $M>0$
and $h_1>0$ such that
\[
    |u(\xi,z)| \leq \frac{M}{|\xi|^{m_0}(1+|\xi|)^{m-m_0}}
                 \phi_N(|\xi|;h_1) \quad 
        \mbox{on $S_I(r) \times D_{R_1}$}.
\]
\end{prop}

   The rest part of this section is used to prove this result.
In subsections 4.1 and 4.2, we present some preparatory discussions
which are needed in the proof of Proposition 4.1, and in 
subsection 4.3 we give a proof of Proposition 4.1.

%
%
% --------------------------------------------------
%    << 4.1. On the functions $\phi_m(x;h)$ >>
% --------------------------------------------------
%

\subsection{On the functions $\phi_m(x;h)$}\label{subsection 4.1}

    In this subsection, let us show some properties of
the functions $\phi_m(x;h)$. We note that 
$\phi_0(x;h)= \exp_q(hx)$ and for $m \geq 1$
\[
    \phi_m(x;h)= \frac{x^{m-1}}{[m-1]_q!} *_q \exp_q(hx)
     = \frac{x^{m-1}}{[m-1]_q!} *_q 
        \sum_{i=0}^{\infty} \frac{h^ix^{i}}{[i]_q!}.
\]

\begin{lem}\label{Lemma4.2}
     Let $1 \leq k \leq n$,
$0<B<h$ and $0<h_0<h$. We have the following results
for $x>0$.
\begin{align}
    &\frac{1}{x^k} \phi_n(x;h) \leq
        \frac{[n-k]_q!}{[n]_q!} 
                               \phi_{n-k}(x;h); \label{4.2}\\
    &\sum_{m=N}^{\infty} B^m \phi_m(x;h)
         \leq \frac{B^N}{1-B/h} \phi_N(x;h); \label{4.3}\\
    &\phi_m(x;h_0) *_q \phi_n(x;h) 
        \leq \frac{1}{1-h_0/h} \phi_{m+n+1}(x;h).\label{4.4}
\end{align}
\end{lem}

\begin{proof}
    (\ref{4.2}) is verified as follows:
\begin{align*}
   \frac{1}{x^k} \phi_n(x;h)
   &= \sum_{i \geq 0} 
            \frac{h^i x^{n-k+i}}{[n-k+i]_q!}
       \times \frac{1}{[n-k+1+i]_q \cdots [n+i]_q} \\
   &\leq \frac{1}{[n-k+1]_q \cdots [n]_q}
      \sum_{i \geq 0} 
            \frac{h^i x^{n-k+i}}{[n-k+i]_q!}
     = \frac{[n-k]_q!}{[n]_q!} 
                               \phi_{n-k}(x;h).
\end{align*}
(\ref{4.3}) is verified as follows:
\begin{align*}
   \sum_{m \geq N} B^m \phi_m(x;h)
   &= \sum_{m \geq N} B^m\sum_{i \geq 0}
                      \frac{h^i x^{m+i}}{[m+i]_q!} 
   = \sum_{k \geq N} \frac{h^k x^{k}}{[k]_q!}
                  \sum_{m+i=k, m \geq N}(B/h)^m \\
   &\leq \sum_{k \geq N}
         \frac{h^k x^{k}}{[k]_q!} \times \frac{(B/h)^N}{1-B/h} 
    = \frac{(B/h)^N}{1-B/h} \sum_{i \geq 0}
                     \frac{h^{N+i} x^{N+i}}{[N+i]_q!} \\
   &= \frac{B^N}{1-B/h} \phi_N(x;h).
\end{align*}
(\ref{4.4}) is verified as follows:
\begin{align*}
   &\phi_m(x;h_0) *_q \phi_n(x;h)
   = \sum_{i \geq 0} \frac{{h_0}^i x^{m+i}}{[m+i]_q!} *_q
      \sum_{j \geq 0} \frac{h^j x^{n+j}}{[n+j]_q!} \\
   &=\sum_{i \geq 0, j \geq 0}
           \frac{{h_0}^ih^j x^{m+n+i+j+1}}{[m+n+i+j+1]_q!} 
   = \sum_{l \geq 0} \frac{h^l x^{m+n+1+l}}{[m+n+1+l]_q!} 
      \sum_{i+j=l} (h_0/h)^i \\
   &\leq \frac{1}{1-h_0/h} \phi_{m+n+1}(x;h).
\end{align*}
\end{proof}

    To see the estimate of $\phi_m(x;h)$, it is enough to use
\[
      \phi_m(x;h) \leq \frac{x^m}{[m]_q!} \exp_q(hx),
      \quad x>0 
\]
and the following result (see Proposition 5.5 in \cite{ramis}).

\begin{lem}\label{Lemma4.3}
    Let $q>1$. We have
\[
    \log (\exp_q(x))
    = \frac{(\log x)^2}{2\log q}
     + \Bigl( - \frac{1}{2}+ \frac{\log(q-1)}{\log q} \Bigr)
         \log x + O(1) 
\]
{\rm (}as $x \longrightarrow +\infty$ in $\BR${\rm )}.
\end{lem}

   As to the estimate of $(a*_qf)(\xi,z)$, by Lemma \ref{Lemma3.3} 
and (\ref{4.4}) of Lemma \ref{Lemma4.2} we have

\begin{lem}\label{Lemma4.4}
    Let $a(\xi,z)$ be a holomorphic function
on $\BC \times D_R$, and let $f(\xi,z)$ be a holomorphic function 
on $S_I(r) \times D_R$. If
\begin{align*}
   &a(\xi,z) \ll A \phi_m(\xi;h_0) \quad \mbox{in $\BC[[\xi]]$
               for any $z \in D_R$}, \\
   &|f(\xi,z)| \leq B \phi_n(|\xi|;h) 
                 \quad \mbox{on $S_I(r) \times D_R$}
\end{align*}
hold for some $A>0$, $h>h_0>0$ and $B>0$, we have
\[
    |(a *_q f)(\xi,z)| \leq \frac{AB}{1-h_0/h}
                \phi_{m+n+1}(|\xi|;h) \quad
        \mbox{on $S_I(qr) \times D_R$}.
\]
\end{lem}

%
%
% --------------------------------------------------
%    << 4.2. On a basic equation >>
% --------------------------------------------------
%

\subsection{On a basic equation}\label{subsection 4.2}

   We set
\[
    {\mathscr H}[w]= P(\xi,z)w+ \sum_{0 \leq i<m_0}
            c_{i,0}(\xi,z) *_q(\xi^iw)
\]
and consider the equation
\begin{equation}
    {\mathscr H}[w] = g(\xi,z).   \label{4.5}
\end{equation}
Let $N$, $h$ and $\beta$ be as in Proposition \ref{Proposition4.1}.  
We have

\begin{lem}\label{Lemma4.5}
    Let $0<R_1 \leq R$, and
let $g(\xi,z) \in {\mathcal O}(S_I(r) \times D_{R_1})$ satisfy 
$|g(\xi,z)| \leq C\phi_n(|\xi|;h)$ on $S_I(r) \times D_{R_1}$
for some $C>0$ and $n \geq N$.  Then, the equation {\rm (\ref{4.5})} 
has a unique solution 
$w(\xi,z) \in {\mathcal O}(S_I(r) \times D_{R_1})$
which satisfies the estimate
\[
    |w(\xi,z)| 
    \leq \frac{C}{\delta(1-\beta)|\xi|^{m_0}(1+|\xi|)^{m-m_0}}
      \phi_n(|\xi|;h) \quad \mbox{on $S_I(r) \times D_{R_1}$}.
\]
\end{lem}

\begin{proof}
    We solve the equation (\ref{4.5}) by the method
of successive approximations.  We set a formal solution
\[
    w(\xi,z)=\sum_{k \geq 0} w_k(\xi,z),
    \quad w_k(\xi,z) \in {\mathcal O}(S_I(r) \times D_{R_1}) 
        \enskip  (k \geq 0)
\]
and determine $w_k(\xi,z)$ ($k \geq 0$) by a solution of the
following system of recursive formulas:
\begin{equation}
           P(\xi,z)w_0=f(\xi,z) \label{4.6}
\end{equation}
and for $k \geq 1$
\begin{equation}
     P(\xi,z)w_k 
     = - \sum_{0 \leq i<m_0} c_{i,0}(\xi,z) *_q (\xi^iw_{k-1}).
            \label{4.7}
\end{equation}
Since $P(\xi,z) \ne 0$ on $S_I(r) \times D_{R_1}$, 
we can determine $w_k(\xi,z) \in {\mathcal O}(S_I(r) \times D_{R_1})$ 
($k=0,1,2,\ldots$) inductively on $k$.
\par
   Let us show the convergence of this formal solution.
By (\ref{4.6}) we have
$w_0(\xi,z)=f(\xi,z)/P(\xi,z)$ and so by the assumption we have
\[
    |w_0(\xi,z)| 
        \leq \frac{C}{\delta |\xi|^{m_0}(1+|\xi|)^{m-m_0}}
        \phi_n(|\xi|;h) \quad \mbox{on $S_I(r) \times D_{R_1}$}.
\]
Since $0 \leq i<m_0$ and $n \geq N$ hold, by (4.2) we have
\begin{align*}
    |\xi^iw_0(\xi,z)| 
    &\leq \frac{C}{\delta |\xi|^{m_0-i}(1+|\xi|)^{m-m_0}}
            \phi_n(|\xi|;h) \\
    &\leq \frac{C}{\delta (1+|\xi|)^{m-m_0}}
     \frac{[n-(m_0-i)]_q!}{[n]_q!} \phi_{n-(m_0-i)}(|\xi|;h) \\
    &\leq \frac{C}{\delta} \frac{1}{[n]_q}
            \phi_{n-(m_0-i)}(|\xi|;h)
          \leq \frac{C}{\delta} \frac{1}{[N]_q}
            \phi_{n-(m_0-i)}(|\xi|;h)
\end{align*}
and so by Lemma \ref{Lemma4.4} we have
\[
    |(c_{i,0}(\xi,z)*_q (\xi^iw_0)|(\xi,z) 
    \leq 
    \frac{C_{i,0}}{(1-h_0/h)} \frac{C}{\delta} \frac{1}{[N]_q}
            \phi_{n}(|\xi|;h) \
        \quad \mbox{on $S_I(r) \times D_{R_1}$}.
\]
Therefore, we have
\[
     \biggl| \sum_{0 \leq i<m_0} 
           c_{i,0}(\xi,z)*_q (\xi^iw_0) \biggr|
     \leq \sum_{0 \leq i<m_0}\frac{C_{i,0}}{(1-h_0/h)}
            \frac{C}{\delta} \frac{1}{[N]_q}
            \phi_{n}(|\xi|;h).
\]
Thus, by (\ref{4.7}) (with $k=1$) and the assumption h${}_2$)
we have
\begin{align*}
    |w_1(\xi,z)| 
    &\leq \frac{1}{\delta |\xi|^{m_0}(1+|\xi|)^{m-m_0}}
          \sum_{0 \leq i<m_0}\frac{C_{i,0}}{(1-h_0/h)}
                             \frac{C}{\delta} \frac{1}{[N]_q}
            \phi_{n}(|\xi|;h) \\
    &= \frac{C \beta}{\delta |\xi|^{m_0}(1+|\xi|)^{m-m_0}}
            \phi_{n}(|\xi|;h)
        \quad \mbox{on $S_I(r) \times D_{R_1}$}.
\end{align*}
\par
    Repeating the same argument as above, we have the estimates
\[
     |w_k(\xi,z)| 
     \leq \frac{C \beta^k}{\delta |\xi|^{m_0}(1+|\xi|)^{m-m_0}}
            \phi_{n}(|\xi|;h)
        \quad \mbox{on $S_I(r) \times D_{R_1}$}
\]
for $k=0,1,2,\ldots$. Since $0 \leq \beta<1$ is supposed, this shows 
that the formal solution is convergent to a true solution $w(\xi,z)$ 
on $S_I(r) \times D_{R_1}$ and it satisfies
\[
    |w(\xi,z)| 
    \leq \frac{C}{\delta(1-\beta) |\xi|^{m_0}(1+|\xi|)^{m-m_0}}
            \phi_{n}(|\xi|;h)
        \quad \mbox{on $S_I(r) \times D_{R_1}$}.
\]
This proves the existence part of Lemma \ref{Lemma4.5}. 
The uniqueness of the solution can be proved in the same way.
\end{proof}

   In subsection 4.3 we will use the norm 
$\|w(\xi)\|_{\rho}= \sup_{z \in D_{\rho}}|w(\xi,z)|$ and 
the following Nagumo's lemma (see Nagumo \cite{nagumo} or 
Lemma 5.1.3 in H\"{o}rmander \cite{hor}).

\begin{lem}\label{Lemma4.6}
   If a holomorphic function $\varphi(z)$ on $D_R$ satisfies
\[
    \| \varphi \|_{\rho} \leq \frac{A}{(R-\rho)^a} \quad
    \mbox{for any $0 < \rho <R$}
\]
for some $A>0$ and $a \geq 0$, we have the estimates
\[
    \bigl\| \partial_{z_i}\varphi \bigr\|_{\rho}
     \leq \frac{(a+1)eA}{(R-\rho)^{a+1}} \quad
     \mbox{for any $0 < \rho <R$ and $i=1,\ldots,d$}.
\]
\end{lem}

%
%
% --------------------------------------------------
%    << 4.3. Proof of Proposition 4.1. >>
% --------------------------------------------------
%

\subsection{Proof of Proposition 4.1}\label{subsection 4.3}

    Let $f(\xi,z)$, $B$, $N$ and $h$ be as in Proposition 
\ref{Proposition4.1}.
First, let us construct a formal solution $u(\xi,z)$ of (\ref{4.1})
in the form
\[
    u(\xi,z)=\sum_{n \geq N}u_n(\xi,z),
      \quad u_n(\xi,z) \in {\mathcal O}(S_I(r) \times D_R)
      \enskip (n \geq N)
\]
so that $u_n(\xi,z)$ ($n \geq N$) are solutions of the 
following recursive formulas:
\begin{equation}
     {\mathscr H}[u_N]= f(\xi,z)  \label{4.8}
\end{equation}
and for $n \geq N+1$
\begin{align}
     {\mathscr H}[u_n]
      = &- \sum_{m_0 \leq i \leq m}
         c_{i,0}(\xi,z) *_q (\xi^i u_{n-1}) \label{4.9} \\
       &\qquad - \sum_{i+\sigma |\alpha| \leq m, |\alpha|>0}
              c_{i,\alpha}(\xi,z) *_q
      (1 *_q (\xi^i \partial_z^{\alpha}u_{n-1})). \notag
\end{align}
We set $L=[m/\sigma]$ (the integer part of $m/\sigma$). 
Let us show

\begin{lem}\label{Lemma4.7} 
    $u_n(\xi,z) \in {\mathcal O}(S_I(r) \times D_R)$ 
{\rm (}$n \geq N${\rm )} 
are uniquely determined inductively on $n$ so that {\rm (\ref{4.8})} 
and {\rm (\ref{4.9})} are satisfied. 
In addition, there are $M>0$ and $H>0$ such that
\begin{equation}
   \| u_n(\xi)\|_{\rho}
   \leq \frac{MH^n}{|\xi|^{m_0}(1+|\xi|)^{m-m_0}(R-\rho)^{L(n-1)}}
         \phi_n(|\xi|;h) 
        \quad \mbox{on $S_I(r)$} \label{4.10}
\end{equation}
holds for any $0<\rho<R$ and any $n \geq N$.
\end{lem}

\begin{proof}
    By applying Lemma \ref{Lemma4.5} to the equation (\ref{4.8}) 
we have a unique solution 
$u_N(\xi,z) \in {\mathcal O}(S_I(r) \times D_R)$ such that
\[
    |u_N(\xi,z)| 
    \leq \frac{B}{\delta(1-\beta)|\xi|^{m_0}(1+|\xi|)^{m-m_0}}
        \phi_N(|\xi|;h) \quad \mbox{on $S_I(r) \times D_R$}.
\]
Therefore, if $M$ and $H$ satisfy 
$MH^N \geq B/(\delta (1-\beta))$ we have (\ref{4.10}) for $n=N$
and any $0<\rho<R$.
\par
   Let us show the general case by induction on $n$. Suppose 
that (\ref{4.10}) is already proved for any $0<\rho<R$. Then, by 
Lemma \ref{Lemma4.6} we have
\begin{align}
   \| \partial_z^{\alpha}u_n(\xi)\|_{\rho}
   &\leq 
    \frac{MH^n e^{|\alpha|}(L(n-1)+1) \cdots (L(n-1)+|\alpha|)}
         {|\xi|^{m_0}(1+|\xi|)^{m-m_0}(R-\rho)^{L(n-1)+|\alpha|}}
         \phi_n(|\xi|;h) \label{4.11}\\
   &\leq \frac{MH^n (eL)^{|\alpha|}n^{|\alpha|}}
         {|\xi|^{m_0}(1+|\xi|)^{m-m_0}(R-\rho)^{Ln}}
         \phi_n(|\xi|;h) \quad \mbox{on $S_I(r)$} \notag
\end{align}
for any $|\alpha| \leq L$ and $0<\rho<R$.
\par
   If $0 \leq i<m_0$ and $|\alpha|>0$, by (\ref{4.11}) and (\ref{4.2}) 
we have
\begin{align*}
   \| \xi^i \partial_z^{\alpha}u_n(\xi)\|_{\rho}
   &\leq \frac{MH^n (eL)^{|\alpha|}n^{|\alpha|}}
         {(R-\rho)^{Ln}} \frac{1}{|\xi|^{m_0-i}} \phi_n(|\xi|;h) \\
   &\leq \frac{MH^n (eL)^{|\alpha|}}
         {(R-\rho)^{Ln}} \frac{n^{|\alpha|}}{[n]_q}
         \phi_{n-(m_0-i)}(|\xi|;h) \\
   &\leq \frac{MH^n (eL)^{|\alpha|} c_0}
         {(R-\rho)^{Ln}}
         \phi_{n-(m_0-i)}(|\xi|;h) \quad \mbox{on $S_I(r)$}
\end{align*}
where
\begin{equation}
   c_0 = \sup_{|\alpha| \leq L, n \geq 1} 
          \frac{n^{|\alpha|}}{[n]_q}
       = \sup_{|\alpha| \leq L, n \geq 1} 
          \frac{n^{|\alpha|}(q-1)}{q^n-1} < \infty. \label{4.12}
\end{equation}
Hence, we have
\[
   \| 1 *_q (\xi^i \partial_z^{\alpha}u_n)\|_{\rho} 
   \leq \frac{MH^n (eL)^{|\alpha|} c_0}
         {(R-\rho)^{Ln}}
         \phi_{n-(m_0-i)+1}(|\xi|;h) \quad \mbox{on $S_I(r)$}
\]
and so in the case $0 \leq i<m_0$ and $|\alpha|>0$ we obtain
\begin{equation}
   \|c_{i,\alpha} *_q 
              (1 *_q (\xi^i \partial_z^{\alpha}u_n))\|_{\rho}
   \leq \frac{C_{i,\alpha}}{1-h_0/h} 
        \frac{MH^n (eL)^{|\alpha|} c_0}{(R-\rho)^{Ln}}
         \phi_{n+1}(|\xi|;h) \label{4.13}
\end{equation}
on $S_I(r)$ for any $0<\rho<R$.
\par
   If $m_0 \leq i \leq m$ and $|\alpha|=0$, by (\ref{4.11}) we have
\[
   \| \xi^i u_n(\xi)\|_{\rho}
   \leq \frac{MH^n}{(R-\rho)^{Ln}} \phi_n(|\xi|;h) 
   \quad \mbox{on $S_I(r)$}
\]
and so we have
\begin{equation}
   \|c_{i,0} *_q (\xi^iu_n)\|_{\rho}
   \leq \frac{C_{i,0}}{1-h_0/h} 
        \frac{MH^n}{(R-\rho)^{Ln}}
    \phi_{n+1}(|\xi|;h) \quad \mbox{on $S_I(r)$} \label{4.14}
\end{equation}
for any $0<\rho<R$
\par
   If $m_0 \leq i \leq m$ and $|\alpha|>0$, by the condition
$i+\sigma |\alpha| \leq m$ we have $i<m$. In this case, 
by (\ref{4.11}) and (\ref{4.2}) we have
\begin{align*}
   \| \xi^i \partial_z^{\alpha}u_n(\xi)\|_{\rho}
   &\leq \frac{MH^n (eL)^{|\alpha|}n^{|\alpha|}}
              {(R-\rho)^{Ln}} 
         \frac{|\xi|^{i-m_0+1}}{(1+|\xi|)^{m-m_0}} 
                  \frac{1}{|\xi|}\phi_n(|\xi|;h) \\
   &\leq \frac{MH^n (eL)^{|\alpha|}}{(R-\rho)^{Ln}} 
         \frac{n^{|\alpha|}}{[n]_q}\phi_{n-1}(|\xi|;h) \\
   &\leq \frac{MH^n (eL)^{|\alpha|} c_0}
         {(R-\rho)^{Ln}}
         \phi_{n-1}(|\xi|;h) \quad \mbox{on $S_I(r)$}
\end{align*}
and so 
\begin{align*}
   \| 1*_q(\xi^i \partial_z^{\alpha}u_n)\|_{\rho}
   &\leq \frac{MH^n (eL)^{|\alpha|} c_0}
         {(R-\rho)^{Ln}}
         \phi_{n}(|\xi|;h) \quad \mbox{on $S_I(r)$}.
\end{align*}
By applying $c_{i,\alpha} *_q$ to this estimate we obtain
\begin{equation}
   \|c_{i,\alpha} *_q 
        ( 1*_q(\xi^i \partial_z^{\alpha}u_n))\|_{\rho}
   \leq \frac{C_{i,\alpha}}{1-h_0/h} 
        \frac{MH^n (eL)^{|\alpha|} c_0}{(R-\rho)^{Ln}}
         \phi_{n+1}(|\xi|;h) \label{4.15}
\end{equation}
on $S_I(r)$ for any $0<\rho<R$
\par
   Thus, by (\ref{4.13}), (\ref{4.14}), (\ref{4.15}) and by setting
$\Lambda=\{(i,\alpha)\,;\, i+\sigma |\alpha| \leq m \}
      \setminus \{(i,0) \,;\, 0 \leq i<m_0 \}$ we have
\begin{align*}
    &\bigl\| \mbox{RHS of $(\ref{4.9})$ 
              (with $n$ replaced by $n+1$)} \bigr\|_{\rho} \\
    &\leq 
      \sum_{(i,\alpha) \in \Lambda}\frac{C_{i,\alpha}}{1-h_0/h} 
        \frac{MH^n (eL)^{|\alpha|} c_0}{(R-\rho)^{Ln}}
         \phi_{n+1}(|\xi|;h) \quad \mbox{on $S_I(r)$}
\end{align*}
and by applying Lemma \ref{Lemma4.5} to the equation 
(\ref{4.9}) (with $n$ replaced by $n+1$) we have a unique solution 
$u_{n+1}(\xi,z) \in {\mathcal O}(S_I(r) \times D_R)$ such
that
\begin{align*}
    \|u_{n+1} \|_{\rho} 
    &\leq \frac{1}{\delta(1-\beta)|\xi|^{m_0}(1+|\xi|)^{m-m_0}} \\
    &\qquad 
       \times
      \sum_{(i,\alpha) \in \Lambda}\frac{C_{i,\alpha}}{1-h_0/h} 
        \frac{MH^n (eL)^{|\alpha|} c_0}{(R-\rho)^{Ln}}
         \phi_{n+1}(|\xi|;h) \quad \mbox{on $S_I(r)$}
\end{align*}
for any $0<\rho<R$.  Thus, if we take $H>0$ sufficiently large
so that 
\[
    H \geq \frac{1}{\delta(1-\beta)}
     \sum_{(i,\alpha) \in \Lambda}
          \frac{C_{i,\alpha}(eL)^{|\alpha|} c_0}{1-h_0/h}
\]
we have (\ref{4.10}) (with $n$ replaced by $n+1$).  This proves 
Lemma \ref{Lemma4.7}.
\end{proof}

   In Lemma \ref{Lemma4.7}, by taking $H>0$ large enough we may 
suppose that $2H \geq h$ holds: then we have $2H/(R-\rho)^L>h$
for any $0<\rho<R$. By Lemma \ref{Lemma4.7} and (\ref{4.3}) 
we have
\begin{align*}
   &\sum_{n \geq N} \|u_n(\xi) \|_{\rho} \\
   &\leq \sum_{n \geq N} 
  \frac{MH^n}{|\xi|^{m_0}(1+|\xi|)^{m-m_0}(R-\rho)^{L(n-1)}}
         \phi_n(|\xi|;h) \\
   &\leq
  \frac{M(R-\rho)^L}{|\xi|^{m_0}(1+|\xi|)^{m-m_0}}
         \sum_{n \geq N} \Bigl( \frac{H}{(R-\rho)^L} 
           \Bigr)^n \phi_n(|\xi|;2H/(R-\rho)^L) \\
   &\leq \frac{M(R-\rho)^L}{|\xi|^{m_0}(1+|\xi|)^{m-m_0}}
           \times 2 \times \Bigl( \frac{H}{(R-\rho)^L} 
           \Bigr)^{N}\phi_N(|\xi|; 2H/(R-\rho)^L).
\end{align*}
This shows that the sum $\sum_{n \geq N}u_n(\xi,z)$ is
convergent to a true solution $u(\xi,z)$ of (\ref{4.1}) on 
$S_I(r) \times D_R$ and it satisfies
\[
    \|u(\xi) \|_{\rho} \leq 
        \frac{2M}{|\xi|^{m_0}(1+|\xi|)^{m-m_0}}
         \frac{H^N}{(R-\rho)^{L(N-1)}} 
        \phi_N \Bigl(|\xi|; \frac{2H}{(R-\rho)^L} \Bigr) 
\]
on $S_I(r)$ for any $0 <\rho<R$. 
This proves the existence part of Proposition \ref{Proposition4.1}. 
\par
   Lastly, let us show the uniqueness of the solution. To do so,
it is enough to prove the following result:

\begin{prop}\label{Proposition4.8}
    Let $0<R_1<R$.  Suppose that 
$u(\xi,z) \in {\mathcal O}(S_I(r) \times D_{R_1})$ satisfies
\begin{align*}
   &P(\xi,z)u + \sum_{0 \leq i \leq m}
        c_{i,0}(\xi,z) *_q (\xi^i u) \\
   &+ \sum_{i+\sigma |\alpha| \leq m, |\alpha|>0}
     c_{i,\alpha}(\xi,z) *_q (1 *_q (\xi^i \partial_z^{\alpha}u))
        = 0 \quad  \mbox{on $S_I(r) \times D_{R_1}$} 
\end{align*}
and 
\[
     |u(\xi,z)| \leq \frac{M}{|\xi|^{m_0}(1+|\xi|)^{m-m_0}}
                 \phi_N(|\xi|;h_1) \quad 
        \mbox{on $S_I(r) \times D_{R_1}$} 
\]
for some $M>0$ and $h_1 (> h_0)$.  Then, we have $u(\xi,z)=0$ on
$S_I(r) \times D_{R_1}$.
\end{prop}

\begin{proof}
    By the same argument as in the proof 
of Lemma \ref{Lemma4.7} we can show that there are $M>0$ and $H>0$ 
such that
\begin{equation}
   \| u(\xi)\|_{\rho}
   \leq \frac{MH^n}{|\xi|^{m_0}(1+|\xi|)^{m-m_0}
                               (R_1-\rho)^{L(n-1)}}
         \phi_n(|\xi|;h_1)  \label{4.16} 
\end{equation}
on  $S_I(r)$ for any $0<\rho<R$ and any $n \geq N$. Since 
$[n+i]_q! \geq [n]_q![i]_q!$ holds, we have
\[
    \phi_n(|\xi|;h_1) 
    \leq \frac{|\xi|^n}{[n]_q!} \exp_q(h_1|\xi|).
\]
Therefore, by (\ref{4.16}) we have
\[
   \| u(\xi)\|_{\rho}
   \leq \frac{M(R-\rho)^L}{|\xi|^{m_0}(1+|\xi|)^{m-m_0}}
        \Bigl( \frac{H}{(R_1-\rho)^{L}} \Bigr)^n
        \frac{|\xi|^n}{[n]_q!} \exp_q(h_1|\xi|)
\]
on $S_I(r)$ for any $0<\rho<R$. Since
$[n]_q! \geq q^{n(n-1)/2}$ (where $p=1/q$) holds, 
by letting $n \longrightarrow \infty$ we obtain 
$\|u(\xi) \|_{\rho}=0$ on $S_I(r)$ for any $0<\rho<R_1$.  
This proves that $u(\xi,z)=0$ holds on $S_I(r) \times D_{R_1}$.
\end{proof}

%
%
% --------------------------------------------------
%          << \S 5. Proof of Theorem 2.3 >>
% --------------------------------------------------
%

\section{Proof of Theorem 2.3}\label{section 5}

   In this section, we will prove Theorem \ref{Theorem2.3}. 
In the next subsection 5.1, we give estimates of the coefficients
$X_n(z)$ ($n \geq 0$) of the formal solution, and in 
subsection 5.2 we prove Theorem \ref{Theorem2.3} by using 
Proposition \ref{Proposition4.1}.

%
%
% --------------------------------------------------
%    << 5.1. Estimates of the formal solution. >>
% --------------------------------------------------
%

\subsection{Estimates of the formal solution}\label{subsection 5.1}

    Let us show

\begin{prop}\label{Proposition5.1}
    Suppose the conditions {\rm (\ref{2.1})} and 
$a_{m_0,0}(0,0) \ne 0$.  Then, if equation {\rm (\ref{1.1})} has a 
formal solution 
$\hat{X}(t,z)= \sum_{n \geq 0}X_n(z)t^n \in {\mathcal O}_{R_0}[[t]]$
{\rm (}with $R_0>0${\rm )}, we can find $0<R<R_0$, $C>0$ and $h>0$ 
such that $|X_n(z)| \leq C h^n[n]_q!$ holds on $D_R$ for any
$n=0,1,2,\ldots$.
\end{prop}

\begin{proof}
     We set $a_{j,0}^0(t,z)=a_{j,0}(t,z)-a_{j,0}(0,z)$ 
for $0 \leq j \leq m_0$, and $a_{j,\alpha}^0(t,z)= a_{j,\alpha}(t,z)$
for $(j,\alpha)$ with $j>m_0$ or $|\alpha|>0$.  We set
\[
      P_1(\lambda;z)= \sum_{0 \leq j \leq m_0}
          a_{j,0}(0,z) \lambda^{j}.
\]
Then, we have $\textrm{ord}_t(a^0_{j,\alpha}) \geq p_{j,\alpha}$
with
\[
    p_{j,\alpha} = 
    \left\{ \begin{array}{ll}
          1, &\mbox{if $0 \leq j \leq m_0$}, \\
    j-m_0, &\mbox{if $m_0 < j \leq m$ and $|\alpha|=0$}, \\
    j-m_0+1, &\mbox{if $m_0< j \leq m$ 
            and $|\alpha|>0$}
    \end{array}
          \right.
\]
and the equation (\ref{1.1}) is expressed in the form
\begin{equation}
    P_1(tD_q;z)X + \sum_{j+\sigma |\alpha| \leq m}
            a_{j,\alpha}^0(t,z) (tD_q)^j \partial_z^{\alpha}X
     = F(t,z).  \label{5.1}
\end{equation}
Since $a_{m_0,0}(0,0) \ne 0$ holds, by taking $R>0$ sufficiently 
small and by taking $N \in \BN^*$ sufficiently large, we can take
$\delta>0$ such that
\begin{equation}
     |P_1([n]_q;z)| \geq \delta (1+[n]_q)^{m_0}
     \quad \mbox{on $D_R$ for any $n \geq N$}.  \label{5.2}
\end{equation}
\par
   We set
\begin{align*}
   &F(t,z)= \sum_{n \geq 0} F_n(z)t^n, \\
   &a_{j,\alpha}^0(t,z)= \sum_{n \geq p_{j,\alpha}}
              a_{j,\alpha,n}(z)t^n \quad
                (j+\sigma |\alpha| \leq m).
\end{align*}
Then, by (\ref{5.1}) we have the relation:
\begin{equation}
   P_1([n]_q; z)X_n = F_n(z)- \sum_{j+\sigma |\alpha| \leq m}
      \sum_{p_{i,\alpha} \leq l \leq n}
         a_{j,\alpha,l}(z)([n-l]_q)^j \partial_z^{\alpha}X_{n-l}
         \label{5.3}
\end{equation}
for any $n=0,1,2,\ldots$.  By taking $R>0$ sufficnetly small we 
may assume that $X_n(z)$ ($n \geq 0$), $F_n(z)$ ($n \geq 0$) and 
$a_{j,\alpha,n}(z)$ ($n \geq p_{j,\alpha}$) are all bounded 
holomorphic functions on $D_R$.  In addition, we may assume that 
$|F_n(z)| \leq Ah^n$ ($n \geq 0$) and 
$|a_{j,\alpha,n}(z)| \leq Ah^n$ ($n \geq p_{j,\alpha}$) hold
on $D_R$ for some $A>0$ and $h>0$. 
\par
   To prove Proposition \ref{Proposition5.1} it is sufficient to 
show the following lemma.

\begin{lem}\label{Lemma5.2}
    Let $L=[m/\sigma]$.  There are $M>0$ and $H>0$ such that
\begin{align}
     \|X_n \|_{\rho} \leq \frac{MH^n[n]_q!}{(R-\rho)^{Ln}}
          \quad \mbox{on $D_R$ for any $0<\rho<R$}
     \label{5.4}
\end{align}
holds for any $n=0,1,2,\ldots$.
\end{lem}

\begin{proof}
    Let $N$ be as in (\ref{5.2}). Since $X_n(z)$ ($0 \leq n \leq N$) 
are bounded holomorphic functions on $D_R$, by taking $M>0$ and $H>0$ 
sufficiently large we may suppose that (\ref{5.4}) is satisfied 
for all $0 \leq n \leq N$.
\par
   Let us show the general case by induction on $n$.
Let $n >N$, and suppose that (\ref{5.4}) (with $n$ replaced by $p$) 
is already proved for all $p<n$.  Then, by applying 
Lemma \ref{Lemma4.6} to the estimate (\ref{5.4}) (with $n$ replaced 
by $n-l$) we have
\begin{align*}
    \|\partial_z^{\alpha}X_{n-l}\|_{\rho}
    &\leq \frac{MH^{n-l}[n-l]_q! \times e^{|\alpha|}(L(n-l)+1) 
             \cdots (L(n-l)+|\alpha|)}
             {(R-\rho)^{L(n-l)+|\alpha|}} \\
    &\leq \frac{MH^{n-l}[n-l]_q! \times 
                      (eL)^{|\alpha|} (n-l+1)^{|\alpha|}}
           {(R-\rho)^{Ln}}
\end{align*}
for any $0<\rho<R$.  Therefore, by (\ref{5.2}), (\ref{5.3}) and the 
condition $0<R-\rho<1$ we have
\begin{align}
    &\|X_n \|_{\rho} \label{5.5} \\
    &\leq \frac{1}{\delta (1+[n]_q)^{m_0}}
          \Bigl[ Ah^n + \sum_{j+\sigma |\alpha| \leq m}
           \sum_{p_{j,\alpha} \leq l \leq n}
         Ah^l([n-l]_q)^j \|\partial_z^{\alpha}X_{n-l}\|_{\rho}
          \Bigr]   \notag \\
    &\leq \frac{1}{\delta (1+[n]_q)^{m_0}}
          \Bigl[ Ah^n + \sum_{j+\sigma |\alpha| \leq m}
           \sum_{p_{j,\alpha} \leq l \leq n}
         Ah^l([n-l]_q)^j \times   \notag \\
    &\qquad \qquad \qquad \qquad \quad \times
      \frac{MH^{n-l}[n-l]_q! \times 
                      (eL)^{|\alpha|} (n-l+1)^{|\alpha|}}
           {(R-\rho)^{Ln}} \Bigr]   \notag \\
    &\leq \frac{MH^n}{\delta (R-\rho)^{Ln}}\Bigl[ \frac{A}{M}
       \Bigl(\frac{h}{H} \Bigr)^n 
             + A \sum_{j+\sigma |\alpha| \leq m}
             K_{j,\alpha}(n)(eL)^{|\alpha|}
           \sum_{p_{j,\alpha} \leq l \leq n}
                \Bigl(\frac{h}{H} \Bigr)^l \Bigr] \notag 
\end{align}
for any $0<\rho<R$, where
\[
    K_{j,\alpha}(n)
    = \frac{([n-p_{j,\alpha}]_q)^j[n-p_{j,\alpha}]_q!
             (n-p_{j,\alpha}+1)^{|\alpha|}}{(1+[n]_q)^{m_0}}.
\]
In the case $0 \leq j \leq m_0$ we have $p_{j,\alpha}=1$ and so
\[
    K_{j,\alpha}(n)
    = \frac{([n-1]_q)^j}{(1+[n]_q)^{m_0}} 
         \frac{(n-1+1)^{|\alpha|}}{[n]_q}\times [n]_q!
       \leq c_0 [n]_q!
\]
where $c_0$ is the one in (\ref{4.12}).
In the case $m_0<j \leq m$ and $|\alpha|=0$ we have 
$p_{j,\alpha}=j-m_0$ and so
\begin{align*}
    K_{j,\alpha}(n)
    &\leq ([n-p_{j,\alpha}]_q)^{j-m_0}[n-p_{j,\alpha}]_q! \\
    &= ([n-(j-m_0)]_q)^{j-m_0} \times [n-(j-m_0)]_q!
    \leq [n]_q! \leq c_0 [n]_q!.
\end{align*}
In the case $m_0<j \leq m$ and $|\alpha|>0$ we have 
$p_{j,\alpha}=j-m_0+1$ and so
\begin{align*}
    K_{j,\alpha}(n)
    &\leq ([n-p_{j,\alpha}]_q)^{j-m_0}[n-p_{j,\alpha}]_q!
              (n-p_{j,\alpha}+1)^{|\alpha|} \\
    &= ([n-(j-m_0+1)]_q)^{j-m_0} \times [n-(j-m_0+1)]_q!\\
    &\hspace*{5cm} \times (n-(j-m_0+1)+1)^{|\alpha|} \\
    &\leq [n-1]_q! \times (n-(j-m_0+1)+1)^{|\alpha|} \\
    &= [n]_q! \times \frac{(n-(j-m_0+1)+1)^{|\alpha|}}{[n]_q}
               \leq c_0 [n]_q!.
\end{align*}
Therefore, by applying these estimates to (\ref{5.5}) we have
\begin{align*}
    \|X_n \|_{\rho}
    &\leq \frac{MH^n c_0 [n]_q!}
               {\delta (R-\rho)^{Ln}}\Bigl[ \frac{A}{M}
       \Bigl(\frac{h}{H} \Bigr)^n 
               + A \sum_{j+\sigma |\alpha| \leq m}
            (eL)^{|\alpha|}
           \sum_{p_{j,\alpha} \leq l \leq n}
                \Bigl(\frac{h}{H} \Bigr)^l \Bigr] \\
    &\leq \frac{MH^n c_0 [n]_q!}{\delta (R-\rho)^{Ln}}
          \Bigl[ \frac{A}{M} \Bigl(\frac{h}{H} \Bigr)^n 
          + A \sum_{j+\sigma |\alpha| \leq m} (eL)^{|\alpha|}
          \frac{(h/H)^{p_{j,\alpha}}}{1-h/H} \Bigr]
\end{align*}
for any $0<\rho<R$.  Thus, if the condition $H>h$ and 
\begin{equation}
   \frac{c_0}{\delta}
   \Bigl[ \frac{A}{M} + A \sum_{j+\sigma |\alpha| \leq m}
            (eL)^{|\alpha|}
                   \frac{(h/H)^{p_{j,\alpha}}}{1-h/H} \Bigr]
    \leq 1  \label{5.6}
\end{equation}
hold, we have the result (\ref{5.4}).  By taking $M>0$ and $H>0$
sufficiently large, we can get the condition (\ref{5.6}). 
This proves Lemma \ref{Lemma5.2}.
\end{proof}

    This completes the proof of Proposition \ref{Proposition5.1}.
\end{proof}

\begin{exa}\label{Example5.3}
     Let us consider 
\begin{equation}
    (tD_q+1)X = \frac{a}{1-z}t + t(t D_q)^2X + bt \partial_z^{\alpha}X,
    \label{5.7}
\end{equation}
where $a>0$, $b>0$ and $\alpha \in \BN^*$. This equation is a 
particular case of equations of type (1.1) with $m_0=1$
and $m=2$. The unique formal solution is given by 
$\hat{X}(t,z)= \sum_{n \geq 1}X_n(z)t^n$ with 
$X_1(z)=a/(([1]_q+1)(1-z))$
and
\[
    X_{n+1}(z)
    = \frac{({[1]_q}^2+b \partial_z^{\alpha}) \cdots 
            ({[n]_q}^2+b \partial_z^{\alpha})}
         {([1]_q+1)([2]_q+1) \cdots ([n+1]_q+1)}
            \Bigl(\frac{a}{(1-z)} \Bigr), \quad n \geq 1.
\]
It is easy to see that 
\[
    X_{n+1}(z) \gg \frac{[n]_q!}{2^{n+1} [n+1]_q} \frac{a}{(1-z)}
          \gg \frac{[n]_q!}{2^{n+1}(1+q)^n} \frac{a}{(1-z)}.
\]
Thus, in the case (\ref{5.7}), we can see that the estimates in 
Proposition \ref{Proposition5.1} is best possible.
\end{exa}

%
%
% --------------------------------------------------
%    << 5.2. Proof of Theorem 2.3. >>
% --------------------------------------------------
%

\subsection{Proof of Theorem 2.3}\label{subsection 5.2}

   Suppose (A${}_1$), (A${}_2$), (A${}_3$) and (\ref{2.2}).
Let $\hat{X}(t,z)= \sum_{n \geq 0}X_n(z)t^n$ be a 
formal solution of (\ref{1.1}). Let $\mu \in \BN^*$ be sufficiently
large and set
\begin{align*}
    X^0(t,z) &= \sum_{n \geq \mu} X_n^0(z)t^{n+1}
     \quad \mbox{with $X_n^0(z)=X_{n+1}(z)$ \,($n \geq \mu$)}, \\
    F^0(t,z) &= F(t,z)- \sum_{j+\sigma |\alpha| \leq m}
              a_{j,\alpha}(t,z) (tD_q)^j \partial_z^{\alpha}
              \sum_{0 \leq n \leq \mu} X_n(z)t^n.
\end{align*}
Then, $X^0(t,z)$ is a formal solution of the equation
\begin{equation}
    \sum_{j+\sigma |\alpha| \leq m}
     a_{j,\alpha}(t,z) (tD_q)^j \partial_z^{\alpha}X^0= F^0(t,z).
    \label{5.8}
\end{equation}

%
%
% -------------------------------------
%  {\bf 5.2.1. Some formulas} 
% -------------------------------------

\subsubsection{Some formulas}\label{subsubsection 5.2.1}

    First, let us show

\begin{lem}\label{Lemma5.4}
   {\rm (1)} For $n \in \BN^*$ we have
\[
      t^n(tD_q)= \frac{1}{q^n} (tD_q-[n]_q)t^n.
\]
\par
   {\rm (2)} For $n \in \BN^*$ and $1 \leq i<n$ we have
\[
     (t^2D_q)t^{n-i} = q^{n-i} t^{n-i}(t^2D_q)+[n-i]_q t^{n-i+1}.
\]
\par
   {\rm (3)} For $n \in \BN^*$ we have
\[
     t^n(tD_q)^n = \frac{1}{q^{n(n-1)/2}} \sum_{i=1}^n
         H_{n,i} t^{n-i}(t^2D_q)^i,
\]
where $H_{n,n}=1$ {\rm (}$n \geq 1${\rm )} and $H_{n,i}$ 
{\rm (}$1 \leq i<n${\rm )} are
constants determined by the recurrence formula: 
\[
    H_{n,i}= q^{n-i}H_{n-1,i-1}+([n-1-i]_q-[n-1]_q)H_{n-1,i}.
\]
\end{lem}

\begin{proof}
   We know that $D_q(f(t)g(t))=D_q(f(t))g(t)+f(qt)D_q(g(t))$ holds.
Hence, we have
\begin{align*}
     (tD_q)(t^nf(t)) &= t([n]_qt^{n-1}f(t)+ (qt)^n D_q(f(t))) \\
     &= [n]_q t^n f(t) + q^n t^n(tD_q)(f(t)),
\end{align*}
that is, $(tD_q)t^n= [n]_q t^n+ q^n t^n(tD_q)$. This leads us to
(1). The result (2) is verified in the same way.
\par
   Let us show (3). The case $n=1$ is clear. Let us show the general
case by induction on $n$. Suppose that (3) is already proved. Then,
by (1) and (2) we have
\begin{align*}
    &t^{n+1}(tD_q)^{n+1} 
     = t(t^n (tD_q))(tD_q)^n
      = t \Bigl(\frac{1}{q^n} (tD_q-[n]_q)t^n \Bigr) (tD_q)^n \\
    &= \frac{1}{q^n} (t^2D_q-[n]_qt) \times 
         \frac{1}{q^{n(n-1)/2}} \sum_{i=1}^n
         H_{n,i} t^{n-i}(t^2D_q)^i \\
    &=\frac{1}{q^{n(n+1)/2}} \biggl[\sum_{i=1}^n H_{n,i}
         \Bigl( q^{n-i} t^{n-i}(t^2D_q)+[n-i]_q t^{n-i+1}
         \Bigr)(t^2D_q)^i\\
   &\qquad \qquad \qquad 
        - \sum_{i=1}^n [n]_q H_{n,i} t^{n-i+1}(t^2D_q)^i \biggr].
\end{align*}
This shows (3) with $n$ replaced by $n+1$. 
\end{proof}

%
%
% -------------------------------------
%  {\bf 5.2.2. A reduction} 
% -------------------------------------

\subsubsection{A reduction}\label{subsubsection 5.2.2}

    We set $b_{j,0}(t,z)=a_{j,0}(t,z)$
(for $0 \leq j<m_0$), $b_{j,0}(t,z)=t^{-(j-m_0)} a_{j,0}(t,z)$ 
(for $m_0 \leq j \leq m$), 
$b_{j,\alpha}(t,z)=t^{-1} a_{j,\alpha}(t,z)$
(for $0 \leq j<m_0$ and $|\alpha|>0$), and 
$b_{j,\alpha}(t,z)=t^{-(j-m_0+2)} a_{j,\alpha}(t,z)$ (for 
$m_0 \leq j<m$ and $|\alpha|>0$). Then, by (\ref{2.1}) and (\ref{2.2}) 
we see that $b_{j,\alpha}(t,z)$ ($j+\sigma |\alpha| \leq m$) are 
holomorphic functions in a neighborhood of 
$(0,0) \in \BC_t \times \BC_z^d$.
\par
   By multiplying (\ref{5.8}) by $t^{m_0}$ we have
\begin{align*}
    &\sum_{0 \leq j<m_0} t^{m_0-j}b_{j,0}(t,z) t^{j}(tD_q)^jX^0
       +\sum_{m_0 \leq j \leq m}b_{j,0}(t,z) t^{j}(tD_q)^jX^0 \\
    &\qquad \qquad + \sum_{0 \leq j<m_0, |\alpha|>0}
           t^{m_0-j+1}b_{j,\alpha}(t,z) t^{j}(tD_q)^j 
           \partial_z^{\alpha}X^0 \\
    &\qquad \qquad + \sum_{m_0 \leq j <m, |\alpha|>0}
           t^{2}b_{j,\alpha}(t,z) t^{j}(tD_q)^j 
           \partial_z^{\alpha}X^0 \\
    &= t^{m_0}F^0(t,z).
\end{align*}
Therefore, by setting 
$b_{j,0}^*(t,z)=t^{m_0-j} b_{j,0}(t,z)$ (for $0 \leq j<m_0$), 
$b_{j,0}^*(t,z)=b_{j,0}(t,z)$ (for $m_0 \leq j \leq m$),
$b_{j,\alpha}^*(t,z)=t^{m_0-j+1} b_{j,\alpha}(t,z)$ (for 
$0 \leq j<m_0$ and $|\alpha|>0$), and 
$b_{j,\alpha}^*(t,z)=t^{2} b_{j,\alpha}(t,z)$ (for 
$m_0 \leq j<m$ and $|\alpha|>0$) we have
\begin{equation}
    \sum_{j+\sigma |\alpha| \leq m}b_{j,\alpha}^*(t,z)
          t^{j}(tD_q)^j \partial_z^{\alpha}X^0
           = t^{m_0}F^0(t,z).  \label{5.9}
\end{equation}
Hence, by (3) of Lemma \ref{Lemma5.4} we have
\[
    \sum_{j+\sigma |\alpha| \leq m}
    b_{j,\alpha}^*(t,z) \frac{1}{q^{j(j-1)/2}}
         \sum_{i=1}^j H_{j,i}t^{j-i} (t^2D_q)^i
         \partial_z^{\alpha}X^0 = t^{m_0}F^0(t,z).
\]
This shows

\begin{lem}\label{Lemma5.5}
    The equation {\rm (\ref{5.9})} can be expressed in the
form
\begin{equation}
     \sum_{i+\sigma |\alpha| \leq m} A_{i,\alpha}(t,z)
           (t^2D_q)^i \partial_z^{\alpha} X^0 
     = t^{m_0}F^0(t,z)  \label{5.10}
\end{equation}
for some holomorphic functions $A_{i,\alpha}(t,z)$ 
{\rm (}$i+\sigma |\alpha| \leq m${\rm )}
in a neighborhood of $(0,0) \in \BC_t \times \BC_z^d$.
In addition, we have
$\mathrm{ord}_t(A_{i,0}) \geq m_0-i$ {\rm (}for $0 \leq i<m_0${\rm )}, 
$\mathrm{ord}_t(A_{i,0}) \geq 0$ {\rm (}for $m_0 \leq i \leq m${\rm )}, 
$\mathrm{ord}_t(A_{i,\alpha}) \geq m_0-i+1$ {\rm (}for 
$0 \leq i<m_0$ and $|\alpha|>0${\rm )}, 
$\mathrm{ord}_t(A_{i,\alpha}) \geq 2$ {\rm (}for 
$m_0 \leq i <m$ and $|\alpha|>0${\rm )}, and
$A_{i,0}(0,z)=b_{i,0}(0,z)/q^{i(i-1)/2}$ {\rm (}for 
$m_0 \leq i \leq m${\rm )}.
\end{lem}

%
%
% -------------------------------------
%   {\bf 5.2.3. $q$-Convolution equation} 
% -------------------------------------

\subsubsection{$q$-Convolution equation}\label{subsubsection 5.2.3}

   By Lemma \ref{Lemma5.5} we see that the equation (\ref{5.10}) 
is written in the form
\begin{align*}
    &\sum_{0 \leq i<m_0} t^{m_0-i}A_{i,0}^0(t,z) (t^{2}D_q)^iX^0 
     \\
    &\qquad 
      +\sum_{m_0 \leq i \leq m} A_{i,0}(0,z)(t^{2}D_q)^iX^0  
     +\sum_{m_0 \leq i \leq m}tA_{i,0}^0(t,z)(t^{2}D_q)^iX^0 \\
    &\qquad + \sum_{0 \leq i<m_0, |\alpha|>0}
           t^{m_0-i+1}A_{i,\alpha}^0(t,z)(t^{2}D_q)^i 
           \partial_z^{\alpha}X^0 \\
    &\qquad + \sum_{m_0 \leq i <m, |\alpha|>0}
           t^{2}A_{i,\alpha}^0(t,z)(t^{2}D_q)^i 
           \partial_z^{\alpha}X^0 = t^{m_0}F^0(t,z)
\end{align*}
for some holomorphic functions $A_{i,\alpha}^0(t,z)$ 
($i+\sigma |\alpha| \leq m$) in a neighborhood of 
$(0,0) \in \BC_t \times \BC_z^d$. We set
\[
    u(\xi,z)= \hat{\mathscr B}_q[X^0](\xi,z)
          = \sum_{n \geq \mu} \frac{X_n^0(z)}{[n]_q!} \xi^n.
\]
By Proposition \ref{Proposition5.1} we know that $u(\xi,z)$ is a 
holomorphic function in a neighborhood of 
$(0,0) \in \BC_{\xi} \times \BC_z^d$. By applying $q$-formal Borel 
transform $\hat{\mathscr B}_q$ to the above equation and by using 
(\ref{3.9}) and (3-4-4) we have
\begin{align*}
    &\sum_{0 \leq i<m_0} {\mathscr B}_q
                  [t^{m_0-i}A_{i,0}^0(t,z)] *_q (\xi^i u) \\
    &\qquad +\sum_{m_0 \leq i \leq m} A_{i,0}(0,z)(\xi^iu) 
          +\sum_{m_0 \leq i \leq m} 
            {\mathscr B}_q[tA_{i,0}^0(t,z)] *_q (\xi^i u) \\
    &\qquad + \sum_{0 \leq i<m_0, |\alpha|>0}
           {\mathscr B}_q[t^{m_0-i}A_{i,\alpha}^0(t,z)] *_q 
                        (1*_q(\xi^i \partial_z^{\alpha}u))\\
    &\qquad + \sum_{m_0 \leq i <m, |\alpha|>0}
           {\mathscr B}_q[tA_{i,\alpha}^0(t,z)] *_q
                        (1*_q(\xi^i \partial_z^{\alpha}u)) \\
    &= {\mathscr B}_q[t^{m_0}F^0(t,z)].
\end{align*}
Thus, by setting 
\begin{align*}
    &c_{i,\alpha}(\xi,z)= {\mathscr B}_q[t^{m_0-i}A_{i,\alpha}^0(t,z)]
        \quad \mbox{(for $0 \leq i<m_0$)}, \\
    &c_{i,\alpha}(\xi,z)= {\mathscr B}_q[tA_{i,\alpha}^0(t,z)]
        \quad \mbox{(for $m_0 \leq i \leq m$)}, \\
    &P(\xi,z)= \sum_{m_0 \leq i \leq m} A_{i,0}(0,z)\xi^i, \\
    &f(\xi,z)= {\mathscr B}_q[t^{m_0}F^0(t,z)]
\end{align*}
we have a $q$-convolution partial differential equation
\begin{align}
    P(\xi,z)u &+ \sum_{0 \leq i \leq m}
        c_{i,0}(\xi,z) *_q (\xi^i u) \label{5.11}\\
   &+ \sum_{i+\sigma |\alpha| \leq m, |\alpha|>0}
     c_{i,\alpha}(\xi,z) *_q (1 *_q (\xi^i \partial_z^{\alpha}u))
        = f(t,z).  \notag
\end{align}
By the definition of $c_{i,\alpha}(\xi,z)$ and $f(\xi,z)$ we see 
that they are holomorphic functions on $\BC_{\xi} \times D_R$ for 
some $R>0$ and we have
\begin{align*}
    &|f(\xi,z)| \leq C \phi_N(|\xi|;h) \quad 
          \mbox{on $\BC \times D_R$ \, \, (with
                $N=m_0+\mu$)}, \\
    &c_{i,\alpha}(\xi,z) \ll C_{i,\alpha} \phi_{m_0-i-1}(\xi;h_0) 
         \quad \mbox{for any $z \in D_R$}
           \quad (0 \leq i<m_0), \\
    &c_{i,\alpha}(\xi,z) \ll C_{i,\alpha} \phi_{0}(\xi;h_0) 
         \quad \mbox{for any $z \in D_R$}
           \quad (m_0 \leq i \leq m)
\end{align*}
for some $C>0$, $h>h_0>0$ and $C_{i,\alpha}>0$ 
($i+\sigma |\alpha| \leq m$). We note:
\[
    P(\xi,z)= \sum_{m_0 \leq i \leq m} A_{i,0}(0,z)\xi^i
           = \sum_{m_0 \leq i \leq m}
                  \frac{b_{i,0}(0,z)}{q^{i(i-1)/2}} \xi^i
          = \xi^{m_0}P_0(\xi,z)
\]
where $P_0(\xi,z)$ is the one appearing in Lemma \ref{Lemma2.2}.

%
%
% -----------------------------------------------------
%   {\bf 5.2.4. Holomorphic extension of $u(\xi,z)$} 
% -----------------------------------------------------

\subsubsection{Holomorphic extension of $u(\xi,z)$}
                   \label{subsubsection 5.2.4}

    Take any $\lambda \in \BC \setminus (\{0\} \cup S)$.
By Lemm \ref{Lemma2.2} we have a $\delta>0$, an interval 
$I=(\theta_1, \theta_2)$ with
$\theta_1<\arg \lambda <\theta_2$ and an $R>0$ such that
\[
    |P(\xi,z)| \geq \delta |\xi|^{m_0}(1+|\xi|)^{m-m_0}
    \quad \mbox{on $S_I \times D_R$}.
\]
Since $\mu$ is taken sufficiently large, we may suppose that
$N=m_0+\mu$ satisfies
\[
     \frac{1}{[N]_q} \sum_{0 \leq i<m_0}
           \frac{C_{i,0}}{\delta (1-h_0/h)} <1.
\]
Thus, we can apply Proposition 4.1 to the equation (\ref{5.11}).
This shows that $u(\xi,z)$ has an analytic extension $u^*(\xi,z)$
to the domain $S_I \times D_{R_1}$ (for some $R_1>0$) as a 
solution of (\ref{5.11}), and we have the estimate
\begin{equation}
    |u^*(\xi,z)| \leq \frac{M_1}{|\xi|^{m_0}(1+|\xi|)^{m-m_0}}
                 \phi_N(|\xi|;h_1) \quad 
        \mbox{on $S_I(r) \times D_{R_1}$}  \label{5.12}
\end{equation}
for some $M_1>0$ and $h_1>0$.

%
%
% -----------------------------------------------------
%   {\bf 5.2.5. Completion of the proof of Theorem 2.3} 
% -----------------------------------------------------

\subsubsection{Completion of the proof of Theorem 2.3}
                   \label{subsubsection 5.2.5}

    To complete the proof of Theorem \ref{Theorem2.3} it is enough 
to show Lemma \ref{Lemma5.6} given below. If this is true, by 
setting
\[
    W(t,z)= \sum_{0 \leq n \leq \mu}X_n(z)t^n
         + {\mathscr L}_q^{\lambda}[u^*](t,z)
\]
we have a true solution of (\ref{1.1}) desired in 
Theorem \ref{Theorem2.3}.

\begin{lem}\label{Lemma5.6}
    There are $C>0$, $h>0$, $A>0$ and $0<B<q$ such that
\begin{align}
   &|u^*(\lambda q^n,z)| \leq Ch^n [n]_q! \quad
      \mbox{on $D_{R_1}$ for $n=0,1,2,\ldots$}, \label{5.13}\\
   &|u^*(\lambda q^{-m},z)| \leq AB^m \quad 
      \mbox{on $D_{R_1}$ for $m=1,2,\ldots$}. \label{5.14}
\end{align}
\end{lem}

\begin{proof}
    Since $u^*(\xi,z)$ is a holomorphic
function in a neighborhood of $(0,0) \in \BC_{\xi} \times \BC_z^d$, 
by setting $B=1$ we have the condition (\ref{5.14}) for a sufficiently 
large $A>0$. By (\ref{5.12}) and Lemma \ref{Lemma4.3} we have
\begin{align*}
    |u^*(\xi,z)| 
     &\leq \frac{M_1}{|\xi|^{m_0}(1+|\xi|)^{m-m_0}} 
         \frac{|\xi|^N}{[N]_q!}
              \exp_q(h_1|\xi|) \\
   &\leq \frac{M_1|\xi|^{N-m_0}}{[N]_q!} \times \\
   &\qquad \times K_1
           \exp \Bigl( \frac{(\log(h_1|\xi|))^2}{2 \log q}
          + \Bigl(-\frac{1}{2}+ \frac{\log(q-1)}{\log q} 
           \Bigr) \log (h_1|\xi|) \Bigr) \\
   &\leq M_2\exp \Bigl( \frac{(\log |\xi|)^2}{2 \log q}
          + \alpha \log |\xi| \Bigr) \quad
        \mbox{on $S_I \times D_{R_1}$} 
\end{align*}
for some $K_1>0$, $M_2>0$ and $\alpha \in \BR$. Hence, we obtain
\[
   |u^*(\lambda q^n,z)|
   \leq M_2 (|\lambda|q^{\alpha})^n q^{n^2/2}
        \exp \Bigl(\frac{(\log |\lambda|)^2}{2 \log q}
              + \alpha \log |\lambda| \Bigr) \quad
      \mbox{on $D_{R_1}$}
\]
for $n=0,1,2,\ldots$. Since $q^{n(n-1)/2} \leq [n]_q!$ 
holds, we have the result (\ref{5.13}).
\end{proof}

%
%
% --------------------------------------------------
%       << \S 6. The case without (\ref{2.2}) >>
% --------------------------------------------------
%

\section{The case without (\ref{2.2})}\label{section 6}

   In Theorem \ref{Theorem2.3}, we have shown the $G_q$-summability
of thr formal solution (\ref{1.2}) under the additional assumption 
(\ref{2.2}). Let us consider here the case without the assumption 
(\ref{2.2}).  We note:

\begin{lem}\label{Lemma6.1}
    Let $f(t)$ be a function in $t$, and
let $n \in \BN^*$. We set $F(\tau)=f(t)$ with $t=\tau^n$; then 
we have
\begin{equation}
      t D_q(f)(t)= \frac{1}{[n]_{q^{1/n}}} 
             \tau D_{q^{1/n}}(F)(\tau).  \label{6.1}
\end{equation}
\end{lem}

\begin{proof}
    By the definition we have
\begin{align*}
    &t D_q(f)(t) = t \times \frac{f(qt)-f(t)}{(q-1)t}
     = \tau^n \times \frac{f(q\tau^n)-f(\tau^n)}{(q-1)\tau^n} \\
    &= \tau^n \times \frac{F(q^{1/n}\tau)-F(\tau)}{(q-1)\tau^n} \\
    &=  \frac{q^{1/n}-1}{(q^{1/n})^n-1} \times \tau
           \times \frac{F(q^{1/n}\tau)-F(\tau)}{(q^{1/n}-1)\tau}
    = \frac{1}{[n]_{q^{1/n}}} 
             \tau D_{q^{1/n}}(F)(\tau).
\end{align*}
\end{proof}

\begin{lem}\label{Lemma6.2}
     Let $n \in \BN^*$: we have 
\begin{equation}
     t D_{q^n}= \frac{q-1}{q^n-1} \sum_{i=0}^{n-1}
        \bigl( (q-1) tD_q+1 \bigr)^i (tD_q).  \label{6.2}
\end{equation}
Note that $D_{q^n}$ in the left-hand side is $q^n$-derivative
and $D_q$ in the right-hand side is $q$-derivative.
\end{lem}

\begin{proof}
    By the definition we have
\begin{align*}
   &tD_{q^n}(f)(t) = t \times \frac{f(q^nt)-f(t)}{(q^n-1)t} \\
   &= t \times \frac{q-1}{q^n-1} \times \frac{(f(q^nt)-f(q^{n-1}t))
      + \cdots +(f(qt)-f(t))}{(q-1)t} \\
   &= \frac{q-1}{q^n-1} \bigl( q^{n-1} tD_q(f)(q^{n-1}t)
          + \cdots+ qtD_q(f)(qt)+tD_q(f)(t) \bigr).
\end{align*}
Therefore, by using the operator $\sigma_q$ defined by
$\sigma_q(f(t))=f(qt)$ we have
\[
    tD_{q^n}(f)(t)
    = \frac{q-1}{q^n-1}\bigl( \sigma_q^{n-1}+ \cdots
              + \sigma_q+1 \bigr) (tD_q)(f)(t).
\]
Since $\sigma_q=(q-1)tD_q+1$ holds, we have (\ref{6.2}).
\end{proof}

\begin{cor}\label{Corollary6.3}
    For any $m \in \BN^*$ and $n \in \BN$ we have
\begin{equation}
      [m]_{q^n}= \frac{1}{[n]_q} \sum_{i=0}^{n-1}
           \bigl( (q-1)[m]_q+1 \bigr)^i [m]_q. 
      \label{6.3}
\end{equation}
\end{cor}

\begin{proof}
   By applying (\ref{6.2}) to $t^m$ we have this result.
\end{proof}

\par
\medskip
   {\it Discussion in the case without {\rm (\ref{2.2})}.}
We set $t=\tau^2$ and $q_1=q^{1/4}$.
Then, by applying Lemmas \ref{Lemma6.1} and \ref{Lemma6.2} we can 
see that our equation (\ref{1.1}) is written in the form
\begin{equation}
    \sum_{j+\sigma |\alpha| \leq m}
            A_{j,\alpha}(\tau,z) \biggl( \frac{1}{[4]_{q_1}}
           \bigl( (q_1-1)(\tau D_{q_1})^2 
      +  2(\tau D_{q_1}) \bigr) \biggr)^j \partial_z^{\alpha}Y
     = G(\tau,z) \label{6.4}
\end{equation}
where 
\begin{align*}
    &A_{j,\alpha}(\tau,z)= a_{j,\alpha}(\tau^2,z) \quad
            (j+\sigma |\alpha| \leq m), \\
    &Y(\tau,z)
          =X(\tau^2,z) = \sum_{n \geq 0} X_n(z) \tau^{2n}, \\
    &G(\tau,z)= F(\tau^2,z).
\end{align*}
In this case, the $t$-Newton polygon $N_t(6.4)$ of (\ref{6.4}) is 
\[
    N_t(6.4)= \{(x,y) \in \BR^2 \,;\, 
               x \leq 2m, y \geq \max\{0, x-2m_0 \} \},
\]
and we have
\begin{align*}
    &\mathrm{ord}_t(A_{j,\alpha}) \geq 
    \left\{ \begin{array}{ll}
     \max \{0, 2j-2m_0 \}, &\mbox{if $|\alpha|=0$}, \\
     \max \{2, 2j-2m_0+2 \}, &\mbox{if $|\alpha|>0$}.
           \end{array}  \right. 
\end{align*}
Therefore, the condition corresponding to (\ref{2.2}) is satisfied.
Thus, we can apply Theorem \ref{Theorem2.3} to the $q_1$-difference 
equation (\ref{6.4}) and we have $G_{q_1}$-summability of the formal 
solution $Y(\tau,z)$.

%
%  --------------------------------
%        << References >>}
%  --------------------------------
%

%
%
%
\par
\bigskip

%---------------------------------------------
%           << Address >>
%---------------------------------------------

\noindent
Hidetoshi Tahara 

\noindent
Department of Information and Communication 
       Sciences, \\
       Sophia University, \\
      Kioicho, Chiyoda-ku, Tokyo 102-8554, Japan.

\noindent
Email: h-tahara@sophia.ac.jp

%
%  --------------------------------
%        << End of Document >>
%  --------------------------------
%
\end{document}